\documentclass{article}
\usepackage{amsfonts}
\usepackage{amsmath}

\newtheorem{theorem}{Theorem}

\newtheorem{corollary}[theorem]{Corollary}

\newtheorem{definition}[theorem]{Definition}

\newtheorem{lemma}[theorem]{Lemma}

\newtheorem{proposition}[theorem]{Proposition}

\newenvironment{proof}[1][Proof]{\noindent\textbf{#1.} }{\ \rule{0.5em}{0.5em}}
\begin{document}

\title{A generalization of $p$-convexity and $q$-concavity on Banach lattices%
}
\author{F. Galaz-Fontes and J.L. Hern\'{a}ndez-Barradas \\
Centro de Investigaci\'{o}n en Matem\'{a}ticas (CIMAT)}
\maketitle

\begin{abstract}
In this paper, considering a real Banach sequence lattice $Y$ instead of a
Lebesgue sequence space $\ell ^{p}$ we generalize $p$-convexity of a linear
operator $T:E\rightarrow X,$ where $E$ is a Banach space and $X$ is a Banach
lattice. Then we prove that basic properties of $p$-convexity remain valid
for $Y$-convex linear operators. Analogous generalizations are given for $q$%
-concavity and $p$-summability and composition properties between these
operators are analyzed.
\end{abstract}

$p$-summability
\\

\section{Introduction}

\bigskip Throughout this work we will consider only real vector spaces.
Recall a linear operator $T$ from a Banach space $E$ into a Banach function
space $X$ (see next section for definitions) is $p$-convex, where $1\leq
p<\infty ,$ if there exists a constant $C>0$ such that%
\begin{equation}
\left\Vert \left\Vert \left( Tw_{1},...,Tw_{n}\right) \right\Vert
_{p}\right\Vert _{X}\leq C\left( \sum_{j=1}^{n}\left\Vert w_{j}\right\Vert
_{E}^{p}\right) ^{1/p},\forall w_{1},...,w_{n}\in E\text{, }\forall n\in 
\mathbb{N}
,  \label{1}
\end{equation}%
where $\left\Vert \left( Tw_{1},...,Tw_{n}\right) \right\Vert _{p}$ is the
function defined by%
\begin{equation}
\left\Vert \left( Tw_{1},...,Tw_{n}\right) \right\Vert _{p}\left( x\right)
=\left( \sum_{j=1}^{n}\left\vert Tw_{j}\left( x\right) \right\vert
^{p}\right) ^{1/p},\forall x\in \Omega .  \notag
\end{equation}%
Naturally, $T$ is $\infty $-convex if there exists a constant $C>0$ such
that 
\begin{equation}
\left\Vert \max_{1\leq j\leq n}\left\{ \left\vert Tw_{j}\right\vert \right\}
\right\Vert _{X}\leq C\max_{1\leq j\leq n}\left\{ \left\Vert
w_{j}\right\Vert _{E}\right\} ,\forall w_{1},...,w_{n}\in E\text{, }\forall
n\in 
\mathbb{N}
,  \label{2}
\end{equation}%
where $\max_{1\leq j\leq n}\left\{ \left\vert Tw_{j}\right\vert \right\} $
is the function defined by%
\begin{equation}
\max_{1\leq j\leq n}\left\{ \left\vert Tw_{j}\right\vert \right\} \left(
x\right) =\max_{1\leq j\leq n}\left\{ \left\vert Tw_{j}\left( x\right)
\right\vert \right\} ,\forall x\in \Omega .  \notag
\end{equation}%
\indent\hspace{-0.05cm}It is well-known that, vias Krivine's functional
calculus introduced in \cite{krivine} $\left( \text{see Theorem \ref{krivin}
below}\right) ,$ expressions of the form $\left( \sum_{j=1}^{n}\left\vert
Tw_{j}\right\vert ^{p}\right) ^{1/p}$ are extended to general Banach
lattices and in this way the inequalities in $\left( \ref{1}\right) $ and $%
\left( \ref{2}\right) $ are also defined in the case that $X$ is a Banach
lattice. In this paper we generalize $p$-convexity by considering a norm in
a Banach sequence lattice $Y$ instead of the norm in $\ell ^{p},$ as we will
now explain.

\bigskip

Let $\left( \Omega ,\Sigma ,\mu \right) $ be a measure space, where $\Omega $
is a non empty set. Take $L^{0}\left( \mu \right) $ to be the vector space
of (equivalence classes of) measurable functions and denote by $\leq $ the $%
\mu $-almost everywhere order on $L^{0}\left( \mu \right) $. Then a $\mu $%
\textit{-Banach function space} ($\mu $-B.f.s.) $Y$ is an ideal of
measurable functions that is also a Banach\ space with an order-preserving
norm $\left\Vert \cdot \right\Vert _{Y}.$ That is, if $f\in L^{0}\left( \mu
\right) $ and $g\in Y$ are such that $\ \left\vert f\right\vert \leq
\left\vert g\right\vert ,$ then $f\in Y$ and $\left\Vert f\right\Vert
_{Y}\leq \left\Vert g\right\Vert _{Y}$. Notice a $\mu $-B.f.s. is a\textit{\
Banach lattice}. That is, a Riesz space with $\left\vert f\right\vert =f\vee
-f,$ that is also a Banach space with respect to an order-preserving norm.

\bigskip

Let us now consider the measure space $\left( 
\mathbb{N}
,2^{%
\mathbb{N}
},\mu _{\#}\right) ,$ where $\mu _{\#}$ is the counting measure. If $Y$ is a 
$\mu _{\#}$-B.f.s. and $e_{n}\in Y,n\in 
\mathbb{N}
$, where $e_{n}$ is the $n$th canonical sequence, we say that $Y$ is a 
\textit{Banach sequence lattice (B.s.l}.). For every finite sequence $\left(
t_{1},...,t_{n}\right) \in 
\mathbb{R}
^{n},n\in 
\mathbb{N}
$ we will write $\left( t_{1},...,t_{n}\right) \in Y$ to express the
sequence $\left\{ t_{1},...,t_{n},0,...\right\} \in Y.$ Also, $\left\Vert
\cdot \right\Vert _{%
\mathbb{R}
^{n},Y}$ will indicate the subspace norm of $Y$ restricted to $%
\mathbb{R}
^{n}.$ Notice that Orlicz sequence spaces and Lebesgue spaces $\ell
^{p},1\leq p\leq \infty $ are B.s.l.. When $Y=\ell ^{p}$ we will write $%
\left\Vert \cdot \right\Vert _{%
\mathbb{R}
^{n},p}$ instead of $\left\Vert \cdot \right\Vert _{%
\mathbb{R}
^{n},\ell ^{p}}$.

\bigskip

We are now in position to present our generalization. Given a Banach lattice 
$X$, a Banach space $E$ and a B.s.l. $Y,$ we will say that a linear operator 
$T:$ $E\rightarrow X$ is $Y$\textit{-convex} if there exists a constant $C>0$
satisfying 
\begin{equation}
\left\Vert \left\Vert \left( Tw_{1},...,Tw_{n}\right) \right\Vert
_{Y}\right\Vert _{X}\leq C\left\Vert \left( \left\Vert w_{1}\right\Vert
_{E},...,\left\Vert w_{n}\right\Vert _{E}\right) \right\Vert _{Y},
\label{20071}
\end{equation}%
for each $n\in 
\mathbb{N}
$ and $w_{1},...,w_{n}\in E.$ Here $\left\Vert \left(
Tw_{1},...,Tw_{n}\right) \right\Vert _{Y}\in X$ is obtained by using
Krivine's functional calculus with the norm $\left\Vert \cdot \right\Vert
_{Y}$ on the B.s.l. $Y.$\newline
\newline
\indent The corresponding definition for $Y$-concavity is now clear. A
linear operator $S:X\rightarrow E$ is $Y$\textit{-concave\ }if there exists
a constant $K>0$ satisfying 
\begin{equation}
\left\Vert \left( \left\Vert Sf_{1}\right\Vert _{E},...,\left\Vert
Sf_{n}\right\Vert _{E}\right) \right\Vert _{Y}\leq K\left\Vert \left\Vert
\left( f_{1},...,f_{n}\right) \right\Vert _{Y}\right\Vert _{X},
\label{20072}
\end{equation}%
for each $n\in 
\mathbb{N}
$ and $f_{1},...,f_{n}\in X.$\newline

\bigskip

Observe that when we take $Y=\ell ^{p},1\leq p\leq \infty $ in the above
definition, conditions $\left( \ref{20071}\right) $ and $\left( \ref{20072}%
\right) $ are reduced to the respective classical $p$-convexity and $p$%
-concavity properties.

\bigskip

Our exposition is divided into 6 sections. After giving an introduction in
Section 1, in Section 2 we present the classical results and definitions
that will be used throughout this work. Given $n\in 
\mathbb{N}
,$ let $\mathcal{H}_{n}$ denote the space of all real-valued continuous
functions $h$ on $%
\mathbb{R}
^{n}$ which are homogeneous of degree 1, that is, $h(\lambda x)=\lambda h(x)$
for all $x\in 
\mathbb{R}
^{n}$ and $\lambda \geq 0$. Let $X$ be a Banach lattice. Then, vias
Krivine's functional calculus, for every $x\in X^{n},$ an operator $\tau
_{n,x}:\mathcal{H}_{n}\rightarrow X$ is defined. So, when fixing $h\in 
\mathcal{H}_{n},$ an operator $\widetilde{h}:X^{n}\rightarrow X$ is
naturally induced by taking $\widetilde{h}\left( x\right) :=\tau
_{n,x}\left( h\right) .$ General properties of such operators are
established in this section.

\bigskip

Many properties for $\tau _{n,x}\left( h\right) ~$have been studied when $%
h=\left\Vert \cdot \right\Vert _{%
\mathbb{R}
^{n},p}$ (\cite[Ch. 16]{Diestel}, \cite[Theorem 1.d.1]{Linden}, \cite[Lemma
2.51]{Okada}). In Section 3 we prove that most of these properties are also
valid for $h=\left\Vert \cdot \right\Vert _{%
\mathbb{R}
^{n},Y}$. This gives us the function $\left\Vert \cdot \right\Vert
_{X^{n},Y,\tau }:X^{n}\rightarrow 
\mathbb{R}
$ defined by%
\begin{equation*}
\left\Vert \left( x_{1},...,x_{n}\right) \right\Vert _{X^{n},Y,\tau
}:=\left\Vert \tau _{n,x}\left( \left\Vert \cdot \right\Vert _{%
\mathbb{R}
^{n},Y}\right) \right\Vert _{X}.
\end{equation*}%
We prove that this function is an order-preserving norm for which $X^{n}$ is
a Banach lattice and that it is equivalent to the natural induced norm $%
\left\Vert \cdot \right\Vert _{X^{n},Y}:X^{n}\rightarrow 
\mathbb{R}
$ given by%
\begin{equation*}
\left\Vert \left( x_{1},...,x_{n}\right) \right\Vert _{X^{n},Y}:=\left\Vert
\left( \left\Vert x_{1}\right\Vert _{X},...,\left\Vert x_{n}\right\Vert
_{X}\right) \right\Vert _{Y}.
\end{equation*}

In Section 4 we use the norms $\left\Vert \cdot \right\Vert _{X^{n},Y}$ and $%
\left\Vert \cdot \right\Vert _{X^{n},Y,\tau }$ to introduce the concept of $%
Y $\textit{-convexity }as in (\ref{20071})\textit{.} Next we study the space
of $Y$-convex operators from a Banach space $E$ into a Banach lattice $X,$ $%
\mathcal{K}^{Y}\left( E,X\right) ,$ and consider its natural norm $%
\left\Vert \cdot \right\Vert _{\mathcal{K}^{Y}}.$ We prove that $\mathcal{K}%
^{Y}\left( E,X\right) $ is a Banach space and that it is continuously
included in the space of linear and continuous operators $\mathcal{L}\left(
E,X\right) $. The section concludes by proving that the space $\mathcal{K}%
^{Y}\left( E,X\right) $ is always non-trivial since it contains every finite
rank operator from $E$ into $X.$ Analogous results are presented for $Y$%
-concavity and the space of $Y$-concave operators, $\mathcal{K}_{Y}\left(
X,E\right) $. Since we have extended many basic properties concerning $p$%
-convexity to $Y$-convexity, it is naturally expected that several results
in recent works related to $p$-convexity can also be generalized, for
example those that appear in \cite{weaktopo}.

\bigskip

In Section 5 we analyze composition properties of $Y$-convex operators.
Given Banach lattices $X,W,$ we introduce the concept of $Y$\textit{-regular
operator }which generalizes that of $p$-regular operator, introduced by
Bukhvalov in \cite{Bukhvalov}, as indicated in \cite{p-regularity}. Given
Banach lattices $X,W,$ a linear operator $T:X\rightarrow W$ is $Y$-regular
if there exists a constant $C>0$ satisfying%
\begin{equation*}
\left\Vert \left( Tx_{1},...,Tx_{n}\right) \right\Vert _{W^{n},Y,\tau }\leq
C\left\Vert \left( x_{1},...,x_{n}\right) \right\Vert _{X^{n},Y,\tau },n\in 
\mathbb{N}
,\left( x_{1},...,x_{n}\right) \in X^{n}.
\end{equation*}%
Let the space of $Y$-regular operators from $X$ into $W$ be denoted by $%
\Lambda _{Y}\left( X,W\right) $ and its natural norm by $\left\Vert \cdot
\right\Vert _{\Lambda _{Y}}$. We prove that $\Lambda _{Y}\left( X,W\right) $
is a Banach space and that it is continuously included in $\mathcal{L}\left(
X,W\right) $. Next, it is proven that $\Lambda _{Y}\left( X,W\right) $
contains every \textit{regular operator,} i.e. difference of two positive
operators, and that $Y$-convexity and $Y$-concavity are preserved under left
and right compositions with $Y$-regular operators. We conclude the section
by showing that $Y$-convexity and $Y$-concavity are invariant under \textit{%
lattice isomorphisms,} i.e., order preserving isomorphisms. It is worthwhile
to recall that $p$-regular operators are examples of $(p,q)$\textit{-regular
operators} when $p=q$ (see $\cite{E.S.P. Tradacete}).$ Although our work
only extends the case $p=q$, the theory developed can be used to define a
generalization for $\left( p,q\right) $-regular operators, considering
B.s.l. $Y_{1}$ and $Y_{2}$ instead of $\ell ^{p}$ and $\ell ^{q}.$

\bigskip

In Section 6 we generalize the important notion of \textit{absolutely }$p$%
\textit{-summing operator.} Given a Banach space $E$ and a Banach lattice $%
X, $ an operator $S:X\rightarrow E$ is \textit{absolutely }$Y$\textit{%
-summing} if there exists a constant $C>0$ such that, for each $n\in 
\mathbb{N}
$ and $x_{1},...,x_{n}\in X,$%
\begin{equation*}
\left\Vert \left( Sx_{1},...,Sx_{n}\right) \right\Vert _{E^{n},Y}\leq
C\sup_{x^{\ast }\in B_{X^{\ast }}}\left\{ \left\Vert \left( \left\langle
x_{1},x^{\ast }\right\rangle ,...,\left\langle x_{n},x^{\ast }\right\rangle
\right) \right\Vert _{Y}\right\} .
\end{equation*}%
Let the space of absolutely $Y$-summing operators from $X$ into $E$ be
denoted by $\Pi _{Y}\left( X,E\right) $ and its natural norm by $\left\Vert
\cdot \right\Vert _{\Pi _{Y}}.$ We prove that the space $\Pi _{Y}\left(
X,E\right) $ contains every finite rank operator from $X$ into $E$ and that $%
\Pi _{Y}\left( X,E\right) $ is a Banach space continuously included in $%
\mathcal{K}_{Y}\left( X,E\right) $. The work concludes by proving that $\Pi
_{Y}\left( X,E\right) $ is an ideal in the space of linear and continuous
operators. As for $p$-regular operators, absolutely $p$-summing operators
are an special case of $(p,q)$\textit{-summing operators} when $p=q$ (see $%
\cite{p-regularity}).$ Again, we only consider the case $p=q,$ but a
generalization for $\left( p,q\right) $-summing operators can also be done,
considering B.s.l. $Y_{1}$ and $Y_{2}$ instead of $\ell ^{p}$ and $\ell
^{q}. $ Therefore, the possibility of obtaining a form of Pietsch%
\'{}%
s factorization theorem and further extensions of several results related to
absolutely summability arise for future study (see \cite{A class of summing}%
, \cite{Q1}, \cite{PMR}, \cite{Mastylo}).

\bigskip

To conclude this introduction we want to mention that, as in the classical
case, the properties of $Y$-convexity and $Y$-concavity can be expressed in
terms of the continuity of an associated linear operator between certain
vector sequence lattices and results of this kind will appear elsewhere.

\section{Preliminary results and Krivine's funcional calculus}

Throughout this work we will indicate the \textit{dual pairing} on a normed
space $V$ by $\left\langle \cdot ,\cdot \right\rangle _{V\times V^{\ast
}}:V\times V^{\ast }\rightarrow 
\mathbb{R}
.$ Also we will denote by $\mathcal{L}\left( V_{1},V_{2}\right) $ the space
of continuous linear operators between two normed spaces $V_{1},V_{2}.$ Some
of the definitions and results presented in this section can be consulted in 
\cite[Ch. 1]{Aliprantis} and \cite[Ch. 1 and 2]{Zaanen1}.

\bigskip

Given a Riesz space $X,$ we will denote the\textit{\ positive cone of }$X$
by $X^{+}$ i.e., 
\begin{equation*}
X^{+}:=\left\{ x\in X:x\geq 0\right\} .
\end{equation*}%
\indent Recall a linear operator between Riesz spaces $P:X\rightarrow Z$ is 
\textit{positive} if $P\left( X^{+}\right) \subset Z^{+}$ (see \cite[Ch. 1]%
{Aliprantis})$.$ Also, a linear operator $T:E\rightarrow F$ is \textit{%
regular} if it can be written as a difference of two positive operators. The
space of all regular operators from $X$ into $Z$ is denoted by $\mathcal{L}%
_{r}\left( X,Z\right) .$

\bigskip

Observe that a positive operator $P:X\rightarrow Z$ satisfies%
\begin{equation}
\left\vert Px\right\vert \leq P\left\vert x\right\vert ,\forall x\in X
\label{positive}
\end{equation}%
and, if in addition $X,Z$ are Banach lattices, then $P:X\rightarrow Z$ is
continuous.

\bigskip

The following well-known property of $\mu $-B.f.s. will be useful for our
presentation \cite[Prop. 2.2]{Okada}.

\begin{theorem}
Let $Y$ be a $\mu $-B.f.s., $\left\{ f_{n}\right\} _{n=1}^{\infty }\subset Y$
and $f\in Y.$ If $f_{n}\overset{Y}{\rightarrow }f$, then there exists a
subsequence $\left\{ f_{n_{k}}\right\} _{k=1}^{\infty }$ such that $%
f_{n_{k}}\rightarrow f$ $\mu $-a.e..\label{convmedida}
\end{theorem}

\begin{corollary}
Let $Y\ $be a $\mu $-B.f.s., $f\in L^{0}(\mu )$ and $\left\{ f_{n}\right\}
_{n=1}^{\infty }\subset Y$ a convergent sequence in $Y.$ If $%
f_{n}\rightarrow f$ $\mu $-a.e., then $f\in Y$ and $f_{n}\overset{Y}{%
\rightarrow }f$.\label{coromuctp}
\end{corollary}

\bigskip

\bigskip We now continue describing Krivine's functional calculus (\cite[Ch.
16]{Diestel}, \cite[Theorem 1.d.1]{Linden}, \cite[Lemma 2.51]{Okada}) and
some of its general results.\bigskip

Given $n\in 
\mathbb{N}
,$ recall $\mathcal{H}_{n}$ is the space of continuous functions $h:%
\mathbb{R}
^{n}\rightarrow 
\mathbb{R}
$ that are homogeneous of degree $1$ i.e., for every $\lambda \geq 0,$ $%
h\left( \lambda t_{1},...,\lambda t_{n}\right) =\lambda h\left(
t_{1},...,t_{n}\right) .$ Observe that every norm on $%
\mathbb{R}
^{n}$ is in $\mathcal{H}_{n},$ in particular each $p$-norm, $\left\Vert
\cdot \right\Vert _{%
\mathbb{R}
^{n},p},1\leq p\leq \infty .$

\bigskip

Let us denote by $S^{n-1}\subset 
\mathbb{R}
^{n}$ the unit sphere with respect to the supremum norm $\left\Vert \cdot
\right\Vert _{%
\mathbb{R}
^{n},\infty },$ given by $\left\Vert \left( a_{1},...,a_{n}\right)
\right\Vert _{%
\mathbb{R}
^{n},\infty }:=\max_{1\leq j\leq n}\left\{ \left\vert a_{j}\right\vert
\right\} ,$ $\left( a_{1},...,a_{n}\right) \in 
\mathbb{R}
^{n}.$ Then we define the norm $\left\Vert \cdot \right\Vert _{\mathcal{H}%
_{n}}$ on $\mathcal{H}_{n}$ by $\left\Vert h\right\Vert _{\mathcal{H}%
_{n}}:=\sup \left\{ \left\vert h\left( t\right) \right\vert :t\in
S^{n-1}\right\} .$ Thus, with the pointwise order, we have that $\mathcal{H}%
_{n}$ is a Banach lattice.

\bigskip

For each $n\in 
\mathbb{N}
$ and $1\leq j\leq n$ we define the $j$th projection $\pi _{n,j}:%
\mathbb{R}
^{n}\rightarrow 
\mathbb{R}
$ by $\pi _{n,j}\left( t_{1},...,t_{n}\right) =t_{j}.$ Clearly $\pi
_{n,j}\in \mathcal{H}_{n},j=1,...,n.$

\begin{theorem}
(Krivine's functional calculus) Let $X$ be a Banach lattice, $n\in 
\mathbb{N}
$ and $x=(x_{1},...,x_{n})\in X^{n}.$ Then, there exists a unique linear map 
$\tau _{n,x}:\label{krivin}$ $\mathcal{H}_{n}\rightarrow X$ such
that\smallskip \newline
\indent$i)$ ${\large \tau }_{n,x}\left( \pi _{n,j}\right) {\large =x}_{j}$
for $1\leq j\leq n.$ \smallskip \newline
\indent$ii)$ ${\Large \tau }_{n,x}$ is order-preserving i.e.,%
\begin{equation*}
\tau _{n,x}\left( h_{1}\vee h_{2}\right) =\tau _{n,x}\left( h_{1}\right)
\vee \tau _{n,x}\left( h_{2}\right) ,\forall h_{1},h_{2}\in \mathcal{H}_{n}.
\end{equation*}
\end{theorem}

\bigskip

Given $h\in $ $\mathcal{H}_{n}$, Krivine's functional calculus determines
the function $\widetilde{h}:X^{n}\rightarrow X$ given by $\widetilde{h}(x):=$
${\large \tau }_{n,x}\left( h\right) $. When there is no risk of confusion
we will write $h$ instead of $\widetilde{h}.$ We will now analyze the
function $h:X^{n}\rightarrow X.$

\begin{definition}
For $n,m\in 
\mathbb{N}
$ we define%
\begin{equation*}
\mathcal{H}_{n}^{m}:=\left\{ G:%
\mathbb{R}
^{n}\rightarrow 
\mathbb{R}
^{m}\text{ s.t. }g_{j}\in \mathcal{H}_{n},\forall \text{ }1\leq j\leq
m\right\} ,
\end{equation*}%
where $g_{j}=\pi _{m,j}\circ G\mathfrak{.}$
\end{definition}

\bigskip

Given $G=\left( g_{1},...,g_{m}\right) \in \mathcal{H}_{n}^{m}$ we have that%
\begin{equation*}
G\left( x\right) :=\left( g_{1}\left( x\right) ,...,g_{m}\left( x\right)
\right) ,\forall x\in X^{n}.
\end{equation*}%
\indent Note that $\mathcal{H}_{n}^{m}$ is a vector space, on which we
define the norm $\left\Vert \cdot \right\Vert _{\mathcal{H}_{n}^{m}}$ by%
\begin{equation*}
\left\Vert G\right\Vert _{\mathcal{H}_{n}^{m}}:=\max_{1\leq j\leq m}\left\{
\left\Vert g_{j}\right\Vert _{\mathcal{H}_{n}}\right\} ,\forall G\in 
\mathcal{H}_{n}^{m}.
\end{equation*}

Let $G\in \mathcal{H}_{n}^{m}.$ Then for each $h\in \mathcal{H}_{m}$ we can
consider the composition $h\circ G:%
\mathbb{R}
^{n}\rightarrow 
\mathbb{R}
$. So we define the operator $J_{G}$ by%
\begin{equation*}
J_{G}\left( h\right) :=h\circ G,\forall h\in \mathcal{H}_{m}.
\end{equation*}%
It follows that 
\begin{equation}
J_{G}\left( h\right) \in \mathcal{H}_{n}.  \label{smt2}
\end{equation}%
Furthermore, given $h_{1},h_{2}\in \mathcal{H}_{m}$ and $\lambda \in 
\mathbb{R}
$ 
\begin{equation*}
J_{G}\left( \lambda h_{1}+h_{2}\right) =\left( \lambda h_{1}+h_{2}\right)
\circ G=\lambda \left( h_{1}\circ g\right) +h_{2}\circ G=\lambda J_{G}\left(
h_{1}\right) +J_{G}\left( h_{2}\right)
\end{equation*}%
and%
\begin{equation*}
J_{G}\left( h_{1}\vee h_{2}\right) =\left( h_{1}\vee h_{2}\right) \circ
G=\left( h_{1}\circ G\right) \vee \left( h_{2}\circ G\right) =J_{G}\left(
h_{1}\right) \vee J_{G}\left( h_{2}\right) .
\end{equation*}%
\indent Hence, $J_{G}:\mathcal{H}_{m}\rightarrow \mathcal{H}_{n}$ is an
order-preserving linear operator.

\begin{lemma}
Let $n,m\in 
\mathbb{N}
$ and $G\in \mathcal{H}_{n}^{m}.$ Then, for each $x=\left(
x_{1},...,x_{n}\right) \in X^{n},$\label{lematmm}\newline
\begin{equation}
{\large \tau }_{m,G\left( x\right) }={\large \tau }_{n,x}\circ J_{G}.
\label{3agos}
\end{equation}
\end{lemma}

\begin{proof}
Since both ${\large \tau }_{n,x}$ and $J_{G}$ are order preserving linear
operators, its composition ${\large \tau }_{n,x}\circ J_{G}:\mathcal{H}%
_{m}\rightarrow X$ is an order preserving linear operator. Next let us fix $%
1\leq j\leq m.$ Then, 
\begin{equation*}
{\large \tau }_{n,x}\circ J_{G}\left( \pi _{m,j}\right) ={\large \tau }%
_{n,x}\left( \pi _{m,j}\circ G\right) ={\large \tau }_{n,x}\left(
g_{j}\right) =g_{j}\left( x\right) ={\large \tau }_{m,G\left( x\right)
}\left( \pi _{m,j}\right) .
\end{equation*}%
\indent From the uniqueness of the operator ${\large \tau }_{m,G\left(
x\right) }$, we conclude that ${\large \tau }_{G\left( x\right) }={\large %
\tau }_{n,x}\circ J_{G}$.
\end{proof}

\bigskip

\bigskip

Let $n,m\in 
\mathbb{N}
$ and, for each $1\leq j\leq m$, take $l_{j}:%
\mathbb{R}
^{n}\rightarrow 
\mathbb{R}
$ to be a linear operator. Then $L:%
\mathbb{R}
^{n}\rightarrow 
\mathbb{R}
^{m}$ and $U:%
\mathbb{R}
^{2n}\rightarrow 
\mathbb{R}
^{n}$ given by%
\begin{equation*}
L\left( t_{1},...,t_{n}\right) :=\left( l_{1}\left( t_{1},...,t_{n}\right)
,...,l_{m}\left( t_{1},...,t_{n}\right) \right) ,\forall \left(
t_{1},...,t_{n}\right) \in 
\mathbb{R}
^{n}
\end{equation*}%
and%
\begin{equation*}
U\left( t_{1},...,t_{2n}\right) =\left( t_{1}\vee t_{n+1},...,t_{n}\vee
t_{2n}\right) ,\forall \left( t_{1},...,t_{2n}\right) \in 
\mathbb{R}
^{2n}
\end{equation*}%
belong to $\mathcal{H}_{n}^{m}$ and $\mathcal{H}_{2n}^{n},$ respectively.
Thus, by the above lemma 
\begin{equation*}
{\large \tau }_{m,L\left( x\right) }={\large \tau }_{n,x}\circ J_{L},\forall
x\in X^{n}
\end{equation*}%
and%
\begin{equation*}
{\large \tau }_{2n,U\left( x\right) }={\large \tau }_{n,x}\circ
J_{U},\forall x\in X^{n}.
\end{equation*}

\bigskip

The next result can be consulted in \cite[Theorem 1.d.1]{Linden}. For the
convenience of the reader we give a proof.

\begin{proposition}
Let $X$ be a Banach lattice, $n\in 
\mathbb{N}
$ and $x=(x_{1},...,x_{n})\in X^{n}.$ Then, for each $h\in \mathcal{H}_{n},%
\label{propkrivi}$%
\begin{equation}
\tau _{n,x}\left( h\right) \leq \left\Vert h\right\Vert _{\mathcal{H}%
_{n}}\max_{1\leq j\leq n}\left\{ \left\vert x_{j}\right\vert \right\} \in X.
\label{nuevaprop}
\end{equation}%
\indent Therefore%
\begin{equation}
\left\Vert \tau _{n,x}\left( h\right) \right\Vert _{X}\leq \left\Vert
h\right\Vert _{\mathcal{H}_{n}}\left\Vert \max_{1\leq j\leq n}\left\{
\left\vert x_{j}\right\vert \right\} \right\Vert _{X}.  \label{NUEVAPROP2}
\end{equation}
\end{proposition}

\begin{proof}
Fix $h\in \mathcal{H}_{n}.$ Let us define the function $l:%
\mathbb{R}
^{n}\rightarrow 
\mathbb{R}
$ by%
\begin{equation*}
l:=\left\Vert h\right\Vert _{\mathcal{H}_{n}}\max_{1\leq j\leq n}\left\{
\left\vert \pi _{n,j}\right\vert \right\} .
\end{equation*}%
That is, for evey $a=\left( a_{1},...,a_{n}\right) \in 
\mathbb{R}
^{n},$%
\begin{equation*}
l\left( a\right) =\left\Vert a\right\Vert _{%
\mathbb{R}
^{n},\infty }\left\Vert h\right\Vert _{\mathcal{H}_{n}}.
\end{equation*}%
\indent Observe that $l\in \mathcal{H}_{n},$ $h\left( 0\right) =0=l\left(
0\right) $ and, for each $a=\left( a_{1},...,a_{n}\right) \in 
\mathbb{R}
^{n}\backslash \left\{ 0\right\} ,$%
\begin{eqnarray*}
\left\vert h\left( a\right) \right\vert &=&\left\vert \left\Vert
a\right\Vert _{%
\mathbb{R}
^{n},\infty }h\left( \frac{a_{1}}{\left\Vert a\right\Vert _{%
\mathbb{R}
^{n},\infty }},...,\frac{a_{n}}{\left\Vert a\right\Vert _{%
\mathbb{R}
^{n},\infty }}\right) \right\vert \\
&& \\
&=&\left\Vert a\right\Vert _{%
\mathbb{R}
^{n},\infty }\left\vert h\left( \frac{a_{1}}{\left\Vert a\right\Vert _{%
\mathbb{R}
^{n},\infty }},...,\frac{a_{n}}{\left\Vert a\right\Vert _{%
\mathbb{R}
^{n},\infty }}\right) \right\vert \\
&& \\
&\leq &\left\Vert a\right\Vert _{%
\mathbb{R}
^{n},\infty }\left\Vert h\right\Vert _{\mathcal{H}_{n}}=l\left( a\right) .
\end{eqnarray*}%
\indent Therefore $\left\vert h\right\vert \leq l$ and consequently%
\begin{eqnarray}
\tau _{n,x}\left( h\right) &\leq &\tau _{n,x}\left( \left\vert h\right\vert
\right) \leq \tau _{n,x}\left( l\right) =\tau _{n,x}\left( \left\Vert
h\right\Vert _{\mathcal{H}_{n}}\max_{1\leq j\leq n}\left\{ \left\vert \pi
_{n,j}\right\vert \right\} \right)  \label{xxx} \\
&&  \notag \\
&=&\left\Vert h\right\Vert _{\mathcal{H}_{n}}\max_{1\leq j\leq n}\left\{
\left\vert x_{j}\right\vert \right\} .  \notag
\end{eqnarray}%
\indent In this way inequality $\left( \ref{nuevaprop}\right) $ is
satisfied. Since $\left\Vert \cdot \right\Vert _{X}$ is a lattice norm, from
inequality $\left( \ref{xxx}\right) ,$ we obtain $\left( \ref{NUEVAPROP2}%
\right) $.
\end{proof}

\begin{lemma}
Let $X$, $W$ be Banach lattices and $T:X\rightarrow W$ an order-preserving
linear operator. Then, for each $n\in 
\mathbb{N}
$ and $x_{1},...,x_{n}\in X,\label{lema1802}$%
\begin{equation}
T\circ \tau _{\left( x_{1,}...,x_{n}\right) }=\tau _{\left(
Tx_{1},...,Tx_{n}\right) }  \label{fr1}
\end{equation}
\end{lemma}

\begin{proof}
Note that $T\circ \tau _{\left( x_{1,}...,x_{n}\right) }$ is an
order-preserving linear operator such that 
\begin{equation*}
\tau _{\left( Tx_{1},...,Tx_{n}\right) }\left( \pi _{n,j}\right)
=Tx_{j}=T\circ \tau _{\left( x_{1,}...,x_{n}\right) }\left( \pi
_{n,j}\right) ,1\leq j\leq n.
\end{equation*}%
\indent Thus, from the uniqueness of the operator $\tau _{\left(
x_{1,}...,x_{n}\right) },$ it follows $\left( \ref{fr1}\right) .$
\end{proof}

\bigskip

Observe that, by above lemma, given $h\in \mathcal{H}_{n}$ and $\lambda \geq
0,\ $%
\begin{equation}
\tau _{\left( x_{1,}...,x_{n}\right) }\left( \lambda h\right) =\lambda \tau
_{\left( x_{1,}...,x_{n}\right) }\left( h\right)  \label{fr3}
\end{equation}

\section{Krivine's functional calculus for $\left\Vert \cdot \right\Vert _{%
\mathbb{R}
^{n},Y}$}

From now on we fix $Y$ to be a Banach sequence lattice and $X,W$ to be
Banach lattices. Notice that, for each $n\in 
\mathbb{N}
,$ $\left\Vert \cdot \right\Vert _{%
\mathbb{R}
^{n},Y}\in \mathcal{H}_{n}.$ Then, for each $x=\left( x_{1},...,x_{n}\right)
\in X^{n}$ we define $\left\Vert \cdot \right\Vert _{Y}:X^{n}\rightarrow X$
by%
\begin{equation*}
\left\Vert \left( x_{1},...,x_{n}\right) \right\Vert _{Y}:=\tau _{n,x}\left(
\left\Vert \cdot \right\Vert _{%
\mathbb{R}
^{n},Y}\right) .
\end{equation*}%
In this section we analyze some properties of the above operator and use it
to construct an order-preserving norm on $X^{n}$.

\bigskip

\begin{lemma}
Let $X$, $W$ be Banach lattices\smallskip .\label{lemalema}\newline
\indent$i)$ If $n\in 
\mathbb{N}
,$ $\lambda \in 
\mathbb{R}
$ and $x=\left( x_{1},...,x_{n}\right) \in X^{n},$ then%
\begin{equation*}
\left\Vert \left( \lambda x_{1},...,\lambda x_{n}\right) \right\Vert
_{Y}=\left\vert \lambda \right\vert \left\Vert \left( x_{1},...,x_{n}\right)
\right\Vert _{Y}.
\end{equation*}%
\indent$ii)$ If $x=\left( x_{1},...,x_{n}\right) ,y=\left(
y_{1},...,y_{n}\right) \in X^{n}$ and $\left\vert x_{j}\right\vert \leq
\left\vert y_{j}\right\vert ,$ $1\leq j\leq n,$ then\label{propo2305}%
\begin{equation*}
\left\Vert \left( x_{1},...,x_{n}\right) \right\Vert _{Y}\leq \left\Vert
\left( y_{1},...,y_{n}\right) \right\Vert _{Y}\in X
\end{equation*}%
and as a consequence%
\begin{equation*}
\left\Vert \left( x_{1},...,x_{n}\right) \right\Vert _{Y}=\left\Vert \left(
\left\vert x_{1}\right\vert ,...,\left\vert x_{n}\right\vert \right)
\right\Vert _{Y}.
\end{equation*}%
\indent$iii)$ If $x=\left( x_{1},...,x_{n}\right) $ and $y=\left(
y_{1},...,y_{n}\right) \in X^{n}$, then\label{lema23052}%
\begin{equation*}
\left\Vert \left( x_{1}+y_{1},...,x_{n}+y_{n}\right) \right\Vert _{Y}\leq
\left\Vert \left( x_{1},...,x_{n}\right) \right\Vert _{Y}+\left\Vert \left(
y_{1},...,y_{n}\right) \right\Vert _{Y}\in X.
\end{equation*}%
\indent Therefore%
\begin{equation}
\left\Vert \left\Vert \left( x_{1}+y_{1},...,x_{n}+y_{n}\right) \right\Vert
_{Y}\right\Vert _{X}\leq \left\Vert \left\Vert \left( x_{1},...,x_{n}\right)
\right\Vert _{Y}\right\Vert _{X}+\left\Vert \left\Vert \left(
y_{1},...,y_{n}\right) \right\Vert _{Y}\right\Vert _{X}.  \label{835}
\end{equation}%
\indent$iv)$ If $x=\left( x_{1},...,x_{n}\right) \in X^{n},$ then\label%
{lema2.51vi} 
\begin{equation*}
\left\Vert \left( x_{1},...,x_{n}\right) \right\Vert _{Y}\leq \left\Vert
e_{1}+...+e_{n}\right\Vert _{%
\mathbb{R}
^{n},Y}\max_{1\leq j\leq n}\left\{ \left\vert x_{j}\right\vert \right\} \in X
\end{equation*}%
and as a consequence%
\begin{equation*}
\left\Vert \left\Vert \left( x_{1},...,x_{n}\right) \right\Vert
_{Y}\right\Vert _{X}\leq \left\Vert e_{1}+...+e_{n}\right\Vert _{%
\mathbb{R}
^{n},Y}\left\Vert \max_{1\leq j\leq n}\left\{ \left\vert x_{j}\right\vert
\right\} \right\Vert _{X}.
\end{equation*}%
\newline
\indent$v)$ Let $T:X\rightarrow W$ be an order-preserving linear operator.
Then%
\begin{equation*}
T\left( \left\Vert \left( x_{1},...,x_{n}\right) \right\Vert _{Y}\right)
=\left\Vert \left( Tx_{1},...,Tx_{n}\right) \right\Vert _{Y},n\in 
\mathbb{N}
,x_{1},...,x_{n}\in X
\end{equation*}
\end{lemma}

\begin{proof}
$i)$ Let $M_{\lambda }:%
\mathbb{R}
^{n}\rightarrow 
\mathbb{R}
^{n}$ be given by $M_{\lambda }\left( t_{1},...,t_{n}\right) :=\left(
\lambda t_{1},...,\lambda t_{n}\right) .$ Clearly $\left\Vert \cdot
\right\Vert _{%
\mathbb{R}
^{n},Y}\circ M_{\lambda }=\left\vert \lambda \right\vert \left\Vert \cdot
\right\Vert _{%
\mathbb{R}
^{n},Y}.$ Then, by the linearity of $\tau _{_{n,x}},$ 
\begin{eqnarray*}
\left\Vert \left( \lambda x_{1},...,\lambda x_{n}\right) \right\Vert _{Y}
&=&\left\Vert \lambda \left( x_{1},...,x_{n}\right) \right\Vert _{Y}=\tau
_{_{n,x}}\left( \left\Vert \cdot \right\Vert _{%
\mathbb{R}
^{n},Y}\circ M_{\lambda }\right) \\
&& \\
&=&\tau _{_{n,x}}\left( \left\vert \lambda \right\vert \left\Vert \cdot
\right\Vert _{%
\mathbb{R}
^{n},Y}\right) =\left\vert \lambda \right\vert \tau _{_{n,x}}\left(
\left\Vert \cdot \right\Vert _{%
\mathbb{R}
^{n},Y}\right) \\
&& \\
&=&\left\vert \lambda \right\vert \left\Vert \left( x_{1},...,x_{n}\right)
\right\Vert _{Y}.
\end{eqnarray*}%
\indent$ii)$ Let $z=\left( x_{1},...,x_{n},y_{1},...,y_{n}\right) \in X^{2n}$
and $L_{1},L_{2}:%
\mathbb{R}
^{2n}\rightarrow 
\mathbb{R}
^{n}$ be given by%
\begin{equation*}
L_{1}\left( t_{1},...,t_{2n}\right) =\left( t_{1},...,t_{n}\right) \text{
and }L_{2}\left( t_{1},...,t_{2n}\right) =\left( \left\vert t_{1}\right\vert
\vee \left\vert t_{n+1}\right\vert ,...,\left\vert t_{n}\right\vert \vee
\left\vert t_{2n}\right\vert \right) .
\end{equation*}%
\indent By Lemma \ref{lematmm},%
\begin{eqnarray*}
\left\Vert \left( x_{1},...,x_{n}\right) \right\Vert _{Y} &=&\tau
_{_{x}}\left( \left\Vert \cdot \right\Vert _{%
\mathbb{R}
^{n},Y}\right) =\tau _{_{L_{1}\left( z\right) }}\left( \left\Vert \cdot
\right\Vert _{%
\mathbb{R}
^{n},Y}\right) =\tau _{_{z}}\left( \left\Vert \cdot \right\Vert _{%
\mathbb{R}
^{n},Y}\circ L_{1}\right) \\
&\leq &\tau _{_{z}}\left( \left\Vert \cdot \right\Vert _{%
\mathbb{R}
^{n},Y}\circ L_{2}\right) =\tau _{_{L_{2}\left( z\right) }}\left( \left\Vert
\cdot \right\Vert _{%
\mathbb{R}
^{n},Y}\right) =\tau _{_{y}}\left( \left\Vert \cdot \right\Vert _{%
\mathbb{R}
^{n},Y}\right) \\
&=&\left\Vert \left( y_{1},...,y_{n}\right) \right\Vert _{Y}.
\end{eqnarray*}%
\indent In particular, by taking $y_{j}=\left\vert x_{j}\right\vert ,$%
\begin{equation*}
\left\Vert \left( x_{1},...,x_{n}\right) \right\Vert _{Y}=\left\Vert \left(
\left\vert x_{1}\right\vert ,...,\left\vert x_{n}\right\vert \right)
\right\Vert _{Y}.
\end{equation*}%
\newline
\newline
\indent$iii)$ Let $z$ and $L_{1}$ be as in $ii)$. Next we define $%
L_{3},L_{4}:%
\mathbb{R}
^{2n}\rightarrow 
\mathbb{R}
^{n}$ by%
\begin{equation*}
L_{3}\left( t_{1},...,t_{2n}\right) =\left( t_{n+1},...,t_{2n}\right) \text{
and }L_{4}=L_{1}+L_{3}.
\end{equation*}%
\indent Then, by Lemma \ref{lematmm} and $\left\Vert \cdot \right\Vert _{Y}$
triangle%
\'{}%
s inequality,%
\begin{eqnarray*}
\left\Vert \left( x_{1}+y_{1},...,x_{n}+y_{n}\right) \right\Vert _{Y}
&=&\tau _{_{x+y}}\left( \left\Vert \cdot \right\Vert _{%
\mathbb{R}
^{n},Y}\right) =\tau _{_{L_{4}\left( z\right) }}\left( \left\Vert \cdot
\right\Vert _{%
\mathbb{R}
^{n},Y}\right) \\
&& \\
&=&\tau _{_{z}}\left( \left\Vert \cdot \right\Vert _{%
\mathbb{R}
^{n},Y}\circ L_{4}\right) \\
&& \\
&\leq &\tau _{_{z}}\left( \left\Vert \cdot \right\Vert _{%
\mathbb{R}
^{n},Y}\circ L_{1}+\left\Vert \cdot \right\Vert _{%
\mathbb{R}
^{n},Y}\circ L_{3}\right) \\
&& \\
&=&\tau _{_{z}}\left( \left\Vert \cdot \right\Vert _{%
\mathbb{R}
^{n},Y}\circ L_{1}\right) +\tau _{_{z}}\left( \left\Vert \cdot \right\Vert _{%
\mathbb{R}
^{n},Y}\circ L_{3}\right) \\
&& \\
&=&\tau _{_{L_{1}\left( z\right) }}\left( \left\Vert \cdot \right\Vert _{%
\mathbb{R}
^{n},Y}\right) +\tau _{_{L_{3}\left( z\right) }}\left( \left\Vert \cdot
\right\Vert _{%
\mathbb{R}
^{n},Y}\right) \\
&& \\
&=&\tau _{_{x}}\left( \left\Vert \cdot \right\Vert _{%
\mathbb{R}
^{n},Y}\right) +\tau _{_{y}}\left( \left\Vert \cdot \right\Vert _{%
\mathbb{R}
^{n},Y}\right) \\
&& \\
&=&\left\Vert \left( x_{1},...,x_{n}\right) \right\Vert _{Y}+\left\Vert
\left( y_{1},...,y_{n}\right) \right\Vert _{Y}.
\end{eqnarray*}%
\indent Finally, from the triangle inequality for $\left\Vert \cdot
\right\Vert _{X},$ the inequality $\left( \ref{835}\right) $ is satisfied$.$%
\newline
\newline
\indent$iv)$ Observe that $\left( a_{1},...,a_{n}\right) \leq $ $%
e_{1}+...+e_{n}\in S^{n-1},$ for any $a\in S^{n-1}$. Then, by Proposition %
\ref{propkrivi}, it follows that, for every $x=\left( x_{1},...,x_{n}\right)
\in X^{n},$%
\begin{eqnarray*}
\left\Vert \left( x_{1},...,x_{n}\right) \right\Vert _{Y} &=&\tau
_{n,x}\left( \left\Vert \cdot \right\Vert _{%
\mathbb{R}
^{n},Y}\right) \leq \left\Vert \left\Vert \cdot \right\Vert _{%
\mathbb{R}
^{n},Y}\right\Vert _{\mathcal{H}_{n}}\max_{1\leq j\leq n}\left\{ \left\vert
x_{j}\right\vert \right\} \\
&& \\
&=&\sup \left\{ \left\vert \left\Vert a\right\Vert _{%
\mathbb{R}
^{n},Y}\right\vert :a\in S^{n-1}\right\} \max_{1\leq j\leq n}\left\{
\left\vert x_{j}\right\vert \right\} \\
&& \\
&\leq &\left\Vert e_{1}+...+e_{n}\right\Vert _{%
\mathbb{R}
^{n},Y}\max_{1\leq j\leq n}\left\{ \left\vert x_{j}\right\vert \right\}
\end{eqnarray*}%
and consequently%
\begin{equation*}
\left\Vert \left\Vert \left( x_{1},...,x_{n}\right) \right\Vert
_{Y}\right\Vert _{X}\leq \left\Vert e_{1}+...+e_{n}\right\Vert _{%
\mathbb{R}
^{n},Y}\left\Vert \max_{1\leq j\leq n}\left\{ \left\vert x_{j}\right\vert
\right\} \right\Vert _{X}.
\end{equation*}%
\newline
\indent$v)$ It follows from Lemma \ref{lema1802} considering $h=\left\Vert
\cdot \right\Vert _{%
\mathbb{R}
^{n},Y}$.\qquad
\end{proof}

\bigskip

Given a Banach space $E$ and $n\in 
\mathbb{N}
$, we will write $\left\Vert \cdot \right\Vert _{E^{n},Y}$ to denote the
norm on $E^{n}$ induced naturally by $\left\Vert \cdot \right\Vert _{%
\mathbb{R}
^{n},Y}$. That is,\newline
\begin{equation}
\left\Vert w\right\Vert _{E^{n},Y}:=\left\Vert \left( \left\Vert
w_{1}\right\Vert _{E},...,\left\Vert w_{n}\right\Vert _{E}\right)
\right\Vert _{%
\mathbb{R}
^{n},Y},\forall w=\left( w_{1},...,w_{n}\right) \in E^{n}.
\label{norma_X^n_Y}
\end{equation}

Note that the norms $\left\Vert \cdot \right\Vert _{E^{n},1}$ and $%
\left\Vert \cdot \right\Vert _{E^{n},Y}$ are equivalent since they are
defined by the equivalent norms $\left\Vert \cdot \right\Vert _{%
\mathbb{R}
^{n},1}$ and $\left\Vert \cdot \right\Vert _{%
\mathbb{R}
^{n},Y}$ respectively. Then, as $\left( E^{n},\left\Vert \cdot \right\Vert
_{E^{n},1}\right) $ is complete it follows that $\left( E^{n},\left\Vert
\cdot \right\Vert _{E^{n},Y}\right) $ is a Banach space. Now, instead of a
Banach space $E$ consider a Banach lattice $X.$ Then $X^{n}$ is a Banach
lattice with the order by components. Furthermore, for every $x=\left(
x_{1},...,x_{n}\right) $, $z=\left( z_{1},...,z_{n}\right) \in X^{n}$ such
that $x\leq z,$%
\begin{equation*}
\left\Vert x\right\Vert _{X^{n},Y}=\left\Vert \left( \left\Vert
x_{1}\right\Vert _{X},...,\left\Vert x_{n}\right\Vert _{X}\right)
\right\Vert _{Y}\leq \left\Vert \left( \left\Vert z_{1}\right\Vert
_{X},...,\left\Vert z_{n}\right\Vert _{X}\right) \right\Vert _{Y}=\left\Vert
z\right\Vert _{X^{n},Y}.
\end{equation*}%
Hence, in this case $\left\Vert \cdot \right\Vert _{X^{n},Y}$ is an
order-preserving norm and consequently\linebreak\ $\left( X^{n},\left\Vert
\cdot \right\Vert _{X^{n},Y}\right) $ is a Banach lattice.

\bigskip

Next, making use of Lemma \ref{lemalema}, we define another norm on the
lattice $X^{n}$. This will be useful to prove the completeness of the spaces
of $Y$-convex and $Y$-concave operators, which we will consider later. The
relation of this norm\ with $\left\Vert \cdot \right\Vert _{X^{n},Y}$ will
be analyzed in the next section.\newline

\begin{definition}
For each $n\in 
\mathbb{N}
$ we define the function $\left\Vert \cdot \right\Vert _{X^{n},Y,\tau
}:X^{n}\rightarrow 
\mathbb{R}
$ by%
\begin{equation*}
\left\Vert x\right\Vert _{X^{n},Y,\tau }:=\left\Vert \left\Vert \left(
x_{1},...,x_{n}\right) \right\Vert _{Y}\right\Vert _{X},\forall x=\left(
x_{1},...,x_{n}\right) \in X^{n}.
\end{equation*}
\end{definition}

Observe that, by Lemma \ref{lemalema}, $\left\Vert \cdot \right\Vert
_{X^{n},Y,\tau }$ is a lattice norm.

\begin{proposition}
The norms $\left\Vert \cdot \right\Vert _{X^{n},1}$ and $\left\Vert \cdot
\right\Vert _{X^{n},Y,\tau }$ are equivalent and consequently $\left(
X^{n},\left\Vert \cdot \right\Vert _{X^{n},Y,\tau }\right) $ is a Banach
lattice.\label{propoca}
\end{proposition}

\begin{proof}
Fix $n\in 
\mathbb{N}
$ and let $x=\left( x_{1},...,x_{n}\right) \in X^{n}.$ By $iv)$ in Lemma \ref%
{lema23052},%
\begin{eqnarray}
\left\Vert x\right\Vert _{X^{n},Y,\tau } &\leq &\left\Vert
e_{1}+...+e_{n}\right\Vert _{Y}\left\Vert \max_{1\leq j\leq n}\left\{
\left\vert x_{j}\right\vert \right\} \right\Vert _{X}  \label{21} \\
&&  \notag \\
&\leq &\left\Vert e_{1}+...+e_{n}\right\Vert _{Y}\left\Vert
\sum_{j=1}^{n}\left\vert x_{j}\right\vert \right\Vert _{X}  \notag \\
&&  \notag \\
&\leq &\left\Vert e_{1}+...+e_{n}\right\Vert _{Y}\sum_{j=1}^{n}\left\Vert
x_{j}\right\Vert _{X}=\left\Vert e_{1}+...+e_{n}\right\Vert _{Y}\left\Vert
x\right\Vert _{X^{n},1}.  \notag
\end{eqnarray}%
\indent On the other hand, fix $1\leq j\leq n$ and define $g_{j}:%
\mathbb{R}
^{n}\rightarrow 
\mathbb{R}
^{n}$ by $g_{j}\left( t_{1},...,t_{n}\right) =\left(
0,...,t_{j},0,...0\right) .$ As $g_{j}\in \mathcal{H}_{n}^{n}$ and%
\begin{equation*}
\left\Vert \cdot \right\Vert _{%
\mathbb{R}
^{n},Y}\circ g_{j}\left( t_{1},...,t_{n}\right) =\left\Vert \left(
0,...,t_{j},0,...0\right) \right\Vert _{%
\mathbb{R}
^{n},Y}=\left\vert t_{j}\right\vert \left\Vert e_{j}\right\Vert _{Y},\forall
\left( t_{1},...,t_{n}\right) \in 
\mathbb{R}
^{n}
\end{equation*}%
we have $\left\Vert \cdot \right\Vert _{%
\mathbb{R}
^{n},Y}\circ g_{j}=\left\Vert e_{j}\right\Vert _{Y}\pi _{n,j}\in H_{n}$.
Then, by Lemma \ref{lematmm} and equality \ref{fr3},%
\begin{eqnarray}
\left\Vert x\right\Vert _{X^{n},Y,\tau } &\geq &\left\Vert \left(
0,...,x_{j},...,0\right) \right\Vert _{X^{n},Y,\tau }=\left\Vert \tau
_{n,g_{j}\left( x\right) }\left( \left\Vert \cdot \right\Vert _{%
\mathbb{R}
^{n},Y}\right) \right\Vert _{X}  \label{2301} \\
&&  \notag \\
&=&\left\Vert \tau _{n,x}\left( \left\Vert \cdot \right\Vert _{%
\mathbb{R}
^{n},Y}\circ g_{j}\right) \right\Vert _{X}=\left\Vert \tau _{n,x}\left(
\left\Vert e_{j}\right\Vert _{Y}\pi _{n,j}\right) \right\Vert _{X}  \notag \\
&&  \notag \\
&=&\left\Vert e_{j}\right\Vert _{Y}\left\Vert \tau _{n,x}\left( \pi
_{n,j}\right) \right\Vert _{X}=\left\Vert e_{j}\right\Vert _{Y}\left\Vert
x_{j}\right\Vert _{X}  \notag
\end{eqnarray}%
and consequently 
\begin{eqnarray}
\left\Vert x\right\Vert _{X^{n},Y,\tau } &\geq &\frac{1}{n}%
\sum_{j=1}^{n}\left\Vert e_{j}\right\Vert _{Y}\left\Vert x_{j}\right\Vert
_{X}\geq \frac{\min \left\{ \left\Vert e_{j}\right\Vert _{Y}\right\} }{n}%
\sum_{j=1}^{n}\left\Vert x_{j}\right\Vert _{X}  \label{22} \\
&=&\frac{\min \left\{ \left\Vert e_{j}\right\Vert _{Y}\right\} }{n}%
\left\Vert x\right\Vert _{X^{n},1}.  \notag
\end{eqnarray}%
\indent From inequalities $\left( \ref{21}\right) $ and $\left( \ref{22}%
\right) $ it follows that $\left\Vert \cdot \right\Vert _{X^{n},1}$ and $%
\left\Vert \cdot \right\Vert _{X^{n},Y,\tau }$ are equivalent norms. Thus $%
\left( X^{n},\left\Vert \cdot \right\Vert _{X^{n},Y,\tau }\right) $ is a
Banach lattice.
\end{proof}

\bigskip

Observe that, when $n=1,$ inequalities $\left( \ref{21}\right) $ and $\left( %
\ref{22}\right) $ imply%
\begin{equation}
\left\Vert x\right\Vert _{X^{1},Y,\tau }=\left\Vert e_{1}\right\Vert
_{Y}\left\Vert x\right\Vert _{X},\forall x\in X.  \label{igualdad812}
\end{equation}%
Also, since $\left\Vert \cdot \right\Vert _{X^{n},1}$ and $\left\Vert \cdot
\right\Vert _{X^{n},Y}$ are equivalent norms, it follows by the above
proposition that $\left\Vert \cdot \right\Vert _{X^{n},Y,\tau }$ and $%
\left\Vert \cdot \right\Vert _{X^{n},Y}$ are equivalent norms.

\bigskip

The following properties are straightforward implications of Lemma \ref%
{lemalema} and the above proposition.

\bigskip

Let $X$ be a Banach lattice and $\left\{ x_{m}\right\} _{m=1}^{\infty
}\subset X^{n}.$ Then$\smallskip $\newline
\indent$i)$ The sequence $\left\{ x_{m}\right\} _{m=1}^{\infty }$ is a
Cauchy sequence in $\left( X^{n},\left\Vert \cdot \right\Vert
_{X^{n},Y}\right) $ if and only if it is a Cauchy sequence in $X$ in each
one of its components.$\smallskip $\newline
\indent$ii)$ The sequence $\left\{ x_{m}\right\} _{m=1}^{\infty }$ is
convergent in $\left( X^{n},\left\Vert \cdot \right\Vert _{X^{n},Y}\right) $
if and only if it is convergent in $X$ in each one of its components. In
this case for $x_{m}:=\left( x_{m,1},...,x_{m,n}\right) ,m\in 
\mathbb{N}
$ we have that%
\begin{equation*}
\lim_{m\rightarrow \infty }x_{m}=\left( \lim_{m\rightarrow \infty
}x_{m,1},...,\lim_{m\rightarrow \infty }x_{m,n}\right)
\end{equation*}%
and%
\begin{equation*}
\left\Vert \lim_{m\rightarrow \infty }x_{m}\right\Vert _{X^{n},Y}=\left\Vert
\left( \lim_{m\rightarrow \infty }\left\Vert x_{m,1}\right\Vert
_{X},...,\lim_{m\rightarrow \infty }\left\Vert x_{m,n}\right\Vert
_{X}\right) \right\Vert _{Y}.
\end{equation*}%
\indent$iii)$ For each $n\in 
\mathbb{N}
,$ 
\begin{equation*}
\left\Vert \left( x_{1},...,x_{n},0\right) \right\Vert
_{X^{n+1},Y}=\left\Vert \left( x_{1},...,x_{n}\right) \right\Vert _{X^{n},Y},
\end{equation*}

\bigskip

\begin{equation*}
\left\Vert \left( x_{1},...,x_{n},0\right) \right\Vert _{X^{n+1},Y,\tau
}=\left\Vert \left( x_{1},...,x_{n}\right) \right\Vert _{X^{n},Y,\tau },
\end{equation*}

\bigskip

\begin{equation*}
\left\Vert \left( x_{1},...,x_{n}\right) \right\Vert _{X^{n},Y}\leq
\left\Vert \left( x_{1},...,x_{n},x_{n+1}\right) \right\Vert _{X^{n+1},Y},
\end{equation*}%
and\newline

\begin{equation*}
\left\Vert \left( x_{1},...,x_{n}\right) \right\Vert _{X^{n},Y,\tau }\leq
\left\Vert \left( x_{1},...,x_{n},x_{n+1}\right) \right\Vert
_{X^{n+1},Y,\tau }.
\end{equation*}

\bigskip

Next we will present a characterization of $\tau _{n,x}\left( \left\Vert
\cdot \right\Vert _{%
\mathbb{R}
^{n},Y}\right) .$ In order to do this, we will identify the dual space $%
\left( 
\mathbb{R}
^{n},\left\Vert \cdot \right\Vert _{Y}\right) ^{\ast }$ with $\left( 
\mathbb{R}
^{n},\left\Vert \cdot \right\Vert _{Y^{\ast }}\right) .$ That is, we will
identify each functional $\varphi \in \left( 
\mathbb{R}
^{n},\left\Vert \cdot \right\Vert _{Y}\right) ^{\ast }$ with the unique
vector $\left( a_{1},...,a_{n}\right) \in 
\mathbb{R}
^{n}$ such that%
\begin{equation*}
\left\langle \left( t_{1},...,t_{n}\right) ,\varphi \right\rangle
=\sum_{j=1}^{n}a_{j}t_{j},\forall \left( t_{1},...,t_{n}\right) \in 
\mathbb{R}
^{n},
\end{equation*}%
and%
\begin{equation*}
\left\Vert \left( a_{1},...,a_{n}\right) \right\Vert _{Y^{\ast }}:=\sup
\left\{ \left\vert \sum_{j=1}^{n}a_{j}t_{j}\right\vert :\left\Vert \left(
t_{1},...,t_{n}\right) \right\Vert _{%
\mathbb{R}
^{n},Y}\leq 1\right\} .
\end{equation*}%
\newline
\indent Thus, for each $t_{1},...,t_{n}\in 
\mathbb{R}
^{n}$ we have that%
\begin{eqnarray}
\left\Vert \left( t_{1},...,t_{n}\right) \right\Vert _{Y} &=&\sup \left\{
\left\vert \left\langle \left( t_{1},...,t_{n}\right) ,\varphi \right\rangle
\right\vert :\left\Vert \varphi \right\Vert _{\left( 
\mathbb{R}
^{n},\left\Vert \cdot \right\Vert _{Y}\right) ^{\ast }}\leq 1\right\}  \notag
\\
&=&\sup \left\{ \left\vert \sum_{j=1}^{n}a_{j}t_{j}\right\vert :\left\Vert
\left( a_{1},...,a_{n}\right) \right\Vert _{%
\mathbb{R}
^{n},Y^{\ast }}\leq 1\right\} .  \label{7nov}
\end{eqnarray}

\begin{lemma}
Let $K$ be a compact space, $\mathcal{C}\subset C\left( K\right) $ a non
empty family and $g\in C\left( K\right) .$ If $g=\sup \mathcal{C}$, then,
there exists a sequence $\left\{ g_{n}\right\} _{n=1}^{\infty }\subset
C\left( K\right) $ such that 
\begin{equation}
\left\Vert f-g_{n}\right\Vert _{C\left( K\right) }\rightarrow 0  \label{jui}
\end{equation}
and each $g_{n}$ is a maximum of functions on \label{2555}$\mathcal{C}.$
\end{lemma}

\begin{proof}
Fix $n\in 
\mathbb{N}
.$ For each $y\in K$ let $f_{y}\in \mathcal{C}$ be such that $g\left(
y\right) -f_{y}\left( y\right) <\frac{1}{n}.$ Consider an open neighbourhood
of $y,$ $V_{y}$, satisfying that $x\in V_{y}$ implies $g\left( x\right)
-f_{y}\left( x\right) <\frac{1}{n}.$ Since $\left\{ V_{y}:y\in K\right\} $
is an open cover of $K$ there exist $y_{1},...,y_{m}\in K$ such that $%
K\subset \cup _{j=1}^{m}V_{y_{j}}.$ Let us define $g_{n}=\max \left\{
f_{y_{j}}:1\leq j\leq m\right\} .$ Since $g\left( x\right) -g_{n}\left(
x\right) <\frac{1}{n},x\in K$ it follows that $\left\{ g_{n}\right\}
_{n=1}^{\infty }$ satisfies $\left( \ref{jui}\right) $.
\end{proof}

\bigskip

The next result extends that of the case $Y=\ell ^{p},1\leq p\leq \infty $ 
\cite[p. 42]{Linden}$.$

\begin{proposition}
For each $n\in 
\mathbb{N}
$ and $x=\left( x_{1},...,x_{n}\right) \in X^{n},\label{lema69}\qquad $%
\begin{equation}
\left\Vert \left( x_{1},...,x_{n}\right) \right\Vert _{Y}=\sup \left\{
\sum_{j=1}^{n}a_{j}x_{j}:\left\Vert \left( a_{1},...,a_{n}\right)
\right\Vert _{Y^{\ast }}\leq 1\right\} .  \label{muero3}
\end{equation}
\end{proposition}

\begin{proof}
Let us fix $n\in 
\mathbb{N}
$ and $x=\left( x_{1},...,x_{n}\right) \in X^{n}.$ Given the
order-preserving property of $\tau _{n,x},$ it follows from $\left( \ref%
{7nov}\right) $ that 
\begin{equation}
\left\Vert \left( x_{1},...,x_{n}\right) \right\Vert _{Y}\geq
\sum_{j=1}^{n}a_{j}x_{j},\forall \left( a_{1},...,a_{n}\right) \in
B_{Y^{\ast }}.  \label{wer}
\end{equation}%
\indent Consider the restriction of $\left\Vert \cdot \right\Vert _{%
\mathbb{R}
^{n},Y}$ to $S^{n-1},$ and let us define the collection%
\begin{equation}
\mathcal{C}=\left\{ f_{a}:S^{n-1}\rightarrow 
\mathbb{R}
\text{ s.t.. }f_{a}\left( t_{1},...,t_{n}\right) =\sum_{j=1}^{n}a_{j}t_{j},%
\text{ where }a=\left( a_{1},...,a_{n}\right) \in B_{Y^{\ast }}.\right\} .
\label{dfdfdf}
\end{equation}%
\newline
\indent Since $S^{n-1}$ is compact, taking $\mathcal{C}=C\left(
S^{n-1}\right) $ and $\left\Vert \cdot \right\Vert _{%
\mathbb{R}
^{n},Y}\in C\left( S^{n-1}\right) ,$ from the above lemma and equality $%
\left( \ref{7nov}\right) ,$ it follows that there exists a sequence $\left\{
g_{m}\right\} _{m=1}^{\infty }\subset C\left( S^{n-1}\right) $ such that $%
\left\Vert \left\Vert \cdot \right\Vert _{%
\mathbb{R}
^{n},Y}-g_{m}\right\Vert _{\infty }\rightarrow 0$ and $g_{m}$ is a maximum
of functions on $\mathcal{C}.$ Now consider $z\in X$ such that%
\begin{equation*}
z\geq \sum_{j=1}^{n}a_{j}x_{j},\forall \left( a_{1},...,a_{n}\right) \in
B_{Y^{\ast }}.
\end{equation*}%
\indent For each $m\in 
\mathbb{N}
,$ $z\geq \tau _{x}\left( g_{m}\right) $ and $\tau _{x}\left( g_{m}\right)
\rightarrow \tau _{x}\left( \left\Vert \cdot \right\Vert _{%
\mathbb{R}
^{n},Y}\newline
\right) =\left\Vert \left( x_{1},...,x_{n}\right) \right\Vert _{Y}.$ Then $%
z\geq \left\Vert \left( x_{1},...,x_{n}\right) \right\Vert _{Y}$ and
consequently $\left( \ref{muero3}\right) $ is satisfied.
\end{proof}

\begin{corollary}
For each $n\in 
\mathbb{N}
$ and $x=\left( x_{1},...,x_{n}\right) \in X^{n},\label{coro1405}$%
\begin{equation}
\left\Vert x\right\Vert _{X^{n},Y,\tau }\leq \sup_{k\in 
\mathbb{N}
}\left\{ \left\Vert
\bigvee\limits_{j=1}^{k}\sum_{i=1}^{n}a_{i,j}x_{i}\right\Vert
_{X}:\left\Vert \left( a_{1,j},...,a_{n,j}\right) \right\Vert _{Y^{\ast
}}\leq 1,1\leq j\leq k\right\} .  \label{wsde}
\end{equation}%
$\qquad $
\end{corollary}

\begin{proof}
Let us fix $n\in 
\mathbb{N}
$ and $x=\left( x_{1},...,x_{n}\right) \in X^{n}.$ For each \thinspace $k\in 
\mathbb{N}
$ let us define%
\begin{equation*}
A_{k}:=\left\{ \bigvee\limits_{j=1}^{k}\sum_{i=1}^{n}a_{i,j}x_{i}:\left\Vert
\left( a_{1,j},...,a_{n,j}\right) \right\Vert _{Y^{\ast }}\leq 1\right\}
\subset X
\end{equation*}%
and%
\begin{equation*}
A:=\bigcup\limits_{k=1}^{\infty }A_{k}\subset X.
\end{equation*}%
\indent By Lemma $\ref{2555}$ there exists a sequence $\left\{ g_{m}\right\}
_{m=1}^{\infty }\subset C\left( S^{n-1}\right) $ such that $\tau _{x}\left(
g_{m}\right) \subset A$ and $\tau _{x}\left( g_{m}\right) \rightarrow
\left\Vert \left( x_{1},...,x_{n}\right) \right\Vert _{Y}.$ Note that%
\begin{equation*}
\sup \left\{ \left\Vert w\right\Vert _{X}:w\in A\right\} \geq \left\Vert
\tau _{x}\left( g_{m}\right) \right\Vert _{X},\forall m\in 
\mathbb{N}%
\end{equation*}%
and consequently%
\begin{equation*}
\sup \left\{ \left\Vert w\right\Vert _{X}:w\in A\right\} \geq
\lim_{m\rightarrow \infty }\left\Vert \tau _{x}\left( g_{m}\right)
\right\Vert _{X}=\left\Vert \left\Vert \left( x_{1},...,x_{n}\right)
\right\Vert _{Y}\right\Vert _{X}=\left\Vert x\right\Vert _{X^{n},Y,\tau }.
\end{equation*}
\end{proof}

\section{$Y$-convexity and $Y$-concavity on Banach lattices}

Let $E,F$ denote Banach spaces, $X,W,Z$ denote Banach lattices and $Y$ stand
for a Banach sequence lattice.

\begin{definition}
A linear operator $T:$ $E\rightarrow X$ is $Y$-convex if there exists a
constant $C>0$ satisfying 
\begin{equation}
\left\Vert \left( Tw_{1},...,Tw_{n}\right) \right\Vert _{X^{n},Y,\tau }\leq
C\left\Vert \left( w_{1},...,w_{n}\right) \right\Vert _{E^{n},Y},
\label{20073}
\end{equation}%
for each $n\in 
\mathbb{N}
$ and $w_{1},...,w_{n}\in E.$ The smallest constant satisfying (\ref{20073})
for all such $n\in N$ and $w_{j}$'s ($j=1,...,n$) is called the $Y$%
-convexity constant of $T$ and is denoted by $M^{Y}\left( T\right) .$ Also,
we will write%
\begin{equation*}
\mathcal{K}^{Y}\left( E,X\right) =\left\{ T:E\rightarrow X\text{ s.t. }T%
\text{ is }Y\text{-convex}\right\} .
\end{equation*}%
\newline
\indent Similarly, a linear operator $S:X\rightarrow E$ is $Y$\textit{%
-concave\ }if there exists a constant $K>0$ satisfying 
\begin{equation}
\left\Vert \left( Sx_{1},...,Sx_{n}\right) \right\Vert _{E^{n},Y}\leq
K\left\Vert \left( x_{1},...,x_{n}\right) \right\Vert _{X^{n},Y,\tau },
\label{20074}
\end{equation}%
for each $n\in 
\mathbb{N}
$ and $x_{1},...,x_{n}\in X.$ The smallest constant satisfying (\ref{20074})
for all such $n\in N$ and $x_{j}$'s ($j=1,...,n$) is called the $Y$%
-concavity constant of $S$ and is denoted by $M_{Y}\left( S\right) .$ Also,
we will write%
\begin{equation*}
\mathcal{K}_{Y}\left( X,E\right) =\left\{ S:X\rightarrow E\text{ s.t. }S%
\text{ is }Y\text{-concave}\right\} .
\end{equation*}
\end{definition}

\begin{lemma}
\label{propoYconvimpliescont}$i)$ $\mathcal{K}^{Y}\left( E,X\right) $ is a
vector subspace of $\mathcal{L}\left( E,X\right) $ and $M^{Y}$ defines a
norm on $\mathcal{K}^{Y}\left( E,X\right) $ such that 
\begin{equation*}
\left\Vert T\right\Vert \leq M^{Y}\left( T\right) ,\forall T\in \mathcal{K}%
^{Y}\left( E,X\right) .
\end{equation*}%
\indent$ii)$ $\mathcal{K}_{Y}\left( X,E\right) $ is a vector subspace of $%
\mathcal{L}\left( X,E\right) $ and $M_{Y}$ defines a norm on $\mathcal{K}%
_{Y}\left( X,E\right) $ such that%
\begin{equation*}
\left\Vert S\right\Vert \leq M_{Y}\left( S\right) ,\forall S\in \mathcal{K}%
_{Y}\left( X,E\right) .
\end{equation*}
\end{lemma}

\begin{proof}
Clearly $\mathcal{K}^{Y}\left( E,X\right) $ is a vector space. Let $T\in 
\mathcal{K}^{Y}\left( E,X\right) $. Then, using \ref{igualdad812}, for each $%
w\in E,$%
\begin{equation*}
\left\Vert e_{1}\right\Vert _{Y}\left\Vert Tw\right\Vert _{X}=\left\Vert
Tw\right\Vert _{X^{1},Y,\tau }\leq M^{Y}\left( T\right) \left\Vert \left(
w\right) \right\Vert _{E^{1},Y}=M^{Y}\left( T\right) \left\Vert
e_{1}\right\Vert _{Y}\left\Vert w\right\Vert _{E}.
\end{equation*}%
\indent Thus, $T$ is continuous and $\left\Vert T\right\Vert \leq
M^{Y}\left( T\right) .$ Hence $\mathcal{K}_{Y}\left( X,E\right) $ is a
vector subspace of $\mathcal{L}\left( X,E\right) .$\newline
\newline
\indent On the other hand, let $T\in \mathcal{K}^{Y}\left( E,X\right) $ be
such that $M^{Y}\left( T\right) =0.$ Then $\left\Vert T\right\Vert \leq
M^{Y}\left( T\right) =0$ and consequently $T=0.$ Since clearly $M^{Y}$
satisfies the triangle inequality$,$ it follows that $M^{Y}$ is a norm on $%
\mathcal{K}^{Y}\left( E,X\right) .$\newline
\newline
\indent The proof of $ii)$ is analogous.
\end{proof}

\bigskip

From now on we will write%
\begin{equation*}
\left\Vert T\right\Vert _{K^{Y}}:=M^{Y}\left( T\right) ,\forall T\in 
\mathcal{K}^{Y}\left( E,X\right)
\end{equation*}%
and%
\begin{equation*}
\left\Vert S\right\Vert _{K_{Y}}:=M_{Y}\left( S\right) ,\forall S\in 
\mathcal{K}_{Y}\left( X,E\right) .
\end{equation*}

\begin{theorem}
\label{tma1405}$\left( \mathcal{K}^{Y}\left( E,X\right) ,\left\Vert \cdot
\right\Vert _{K^{Y}}\right) $ and $\left( \mathcal{K}_{Y}\left( X,E\right)
,\left\Vert \cdot \right\Vert _{K_{Y}}\right) $ are Banach spaces$.$
\end{theorem}

\begin{proof}
Let $\left\{ T_{m}\right\} _{n=1}^{\infty }\subset \mathcal{K}^{Y}\left(
E,X\right) $ such that $\sum_{m=1}^{\infty }\left\Vert T_{m}\right\Vert
_{K^{Y}}=M<\infty .$Then, for each $w\in E,$ 
\begin{equation*}
\sum_{m=1}^{\infty }\left\Vert T_{m}\left( w\right) \right\Vert _{X}\leq
\sum_{m=1}^{\infty }\left\Vert T_{m}\right\Vert \left\Vert w\right\Vert
_{E}\leq \sum_{m=1}^{\infty }\left\Vert T_{m}\right\Vert _{K^{Y}}\left\Vert
w\right\Vert _{E}\leq M\left\Vert w\right\Vert _{E}
\end{equation*}%
and, as $X$ is complete$,$ $\sum_{m=1}^{\infty }T_{m}\left( w\right) $
converges. Let us define the linear operator $T:E\rightarrow X$ by 
\begin{equation}
T\left( w\right) :=\sum_{m=1}^{\infty }T_{m}\left( w\right) ,\forall w\in E.
\label{dff}
\end{equation}%
\newline
\indent Let $n\in 
\mathbb{N}
$ and $w_{1},...,w_{n}\in E.$ Then%
\begin{eqnarray*}
\left\Vert \left( Tw_{1},...,Tw_{n}\right) \right\Vert _{X^{n},Y,\tau }%
\hspace{-0.2cm} &=&\left\Vert \left( \lim_{k\rightarrow \infty
}\sum_{m=1}^{k}T_{m}\left( w_{1}\right) ,...,\lim_{k\rightarrow \infty
}\sum_{m=1}^{k}T_{m}\left( w_{n}\right) \right) \right\Vert _{X^{n},Y,\tau }
\\
&& \\
&=&\lim_{k\rightarrow \infty }\left\Vert \left( \sum_{m=1}^{k}T_{m}\left(
w_{1}\right) ,...,\sum_{m=1}^{k}T_{m}\left( w_{n}\right) \right) \right\Vert
_{X^{n},Y,\tau } \\
&& \\
&\leq &\lim_{k\rightarrow \infty }\sum_{m=1}^{k}\left\Vert \left(
T_{m}w_{1},...,T_{m}w_{n}\right) \right\Vert _{X^{n},Y,\tau } \\
&& \\
&\leq &\sum_{m=1}^{\infty }\left\Vert T_{m}\right\Vert _{K^{Y}}\left\Vert
\left( w_{1},...,w_{n}\right) \right\Vert _{E^{n},Y} \\
&& \\
&=&M\left\Vert \left( w_{1},...,w_{n}\right) \right\Vert _{E^{n},Y}.
\end{eqnarray*}%
\indent Thus $T\in \mathcal{K}\left( E,X\right) $ and consequently $\mathcal{%
K}^{Y}\left( E,X\right) $ is complete.\newline
\newline
\indent The proof for $\mathcal{K}_{Y}\left( X,E\right) $ is analogous.
\end{proof}

\bigskip

Recall that, given Banach spaces $E,F,$ an operator $T\in \mathcal{L}\left(
E,F\right) $ is called a \textit{finite rank operator} if its image has
finite dimension. This is equivalent to the existence of $n\in 
\mathbb{N}
,$ $\varphi _{1},...,\varphi _{n}\in E^{\ast }$ and $w_{1},...,w_{n}\in F$
such that%
\begin{equation*}
T\left( x\right) =\sum\limits_{j=1}^{n}\varphi _{j}\left( x\right) \cdot
w_{j},\forall x\in E.
\end{equation*}%
\indent When $n=1$ we say that $T$ is a \textit{rank }$1$\textit{\ operator}$%
.$

\begin{proposition}
Let $X$ be a Banach lattice and $E$ a Banach space. Then every finite rank
operator $T:E\rightarrow X$ is $Y$-convex.
\end{proposition}

\begin{proof}
Let $T:E\rightarrow X$ be a rank $1$ operator. Take $\varphi \in E^{\ast }$
and $x\in X$ such that $T\left( w\right) =\left\langle w,\varphi
\right\rangle x,w\in E.$ Then, for every $n\in 
\mathbb{N}
$ and $w_{1},...,w_{n}\in E$,%
\begin{eqnarray}
\left\Vert \left( Tw_{1},...,Tw_{n}\right) \right\Vert _{X^{n},Y,\tau }
&=&\left\Vert \left( \left\langle w_{1},\varphi \right\rangle
x,...,\left\langle w_{n},\varphi \right\rangle x\right) \right\Vert
_{X^{n},Y,\tau }  \label{trt} \\
&&  \notag \\
&\leq &\left\Vert \varphi \right\Vert \left\Vert \left( \left\Vert
w_{1}\right\Vert _{E}\left\vert x\right\vert ,...,\left\Vert
w_{n}\right\Vert _{E}\left\vert x\right\vert \right) \right\Vert
_{X^{n},Y,\tau }  \notag \\
&&  \notag \\
&=&\left\Vert \varphi \right\Vert \left\Vert \left\Vert \left( \left\Vert
w_{1}\right\Vert _{E}\left\vert x\right\vert ,...,\left\Vert
w_{n}\right\Vert _{E}\left\vert x\right\vert \right) \right\Vert
_{Y}\right\Vert _{X}  \notag
\end{eqnarray}%
and by Proposition \ref{lema69}%
\begin{eqnarray}
\left\Vert \left( \left\Vert w_{1}\right\Vert _{E}\left\vert x\right\vert
,...,\left\Vert w_{n}\right\Vert _{E}\left\vert x\right\vert \right)
\right\Vert _{Y} &=&\sup \left\{ \sum_{j=1}^{n}a_{j}\left\Vert
w_{j}\right\Vert _{E}\left\vert x\right\vert :\left\Vert \left(
a_{1},...,a_{n}\right) \right\Vert _{Y^{\ast }}\leq 1\right\}  \notag \\
&&  \notag \\
&\leq &\left\vert x\right\vert \sup \left\{ \sum_{j=1}^{n}a_{j}\left\Vert
w_{j}\right\Vert _{E}:\left\Vert \left( a_{1},...,a_{n}\right) \right\Vert
_{Y^{\ast }}\leq 1\right\}  \notag \\
&&  \notag \\
&=&\left\vert x\right\vert \left\Vert \left( \left\Vert w_{1}\right\Vert
_{E},...,\left\Vert w_{n}\right\Vert _{E}\right) \right\Vert _{Y}  \notag \\
&&  \notag \\
&=&\left\vert x\right\vert \left\Vert \left( w_{1},...,w_{n}\right)
\right\Vert _{E^{n},Y}.  \label{trt2}
\end{eqnarray}%
\indent Thus, by $\left( \ref{trt}\right) $ and $\left( \ref{trt2}\right) ,$%
\begin{eqnarray*}
\left\Vert \left( Tw_{1},...,Tw_{n}\right) \right\Vert _{X^{n},Y,\tau }
&\leq &\left\Vert \varphi \right\Vert \left\Vert \left\vert x\right\vert
\left\Vert \left( w_{1},...,w_{n}\right) \right\Vert _{E^{n},Y}\right\Vert
_{X} \\
&& \\
&=&\left\Vert \varphi \right\Vert \left\Vert x\right\Vert _{X}\left\Vert
\left( w_{1},...,w_{n}\right) \right\Vert _{E^{n},Y}.
\end{eqnarray*}%
\indent Therefore $T$ is $Y$-convex. The result follows since every finite
rank operator can be expressed as a linear combination of rank $1$ operators.
\end{proof}

\bigskip

In Section 6 we will prove that finite rank operators are also $Y$-concave.

\section{Composition of $Y$-convex and $Y$-concave\protect\linebreak %
operators}

In this section we will analyze conditions under which a composition of
operators is $Y$-convex or $Y$-concave. For this we present a generalization
of the classical space $\Lambda _{p}$, consisting of $p$-regular operators (%
\cite{Bukhvalov}, \cite[p. 79]{Okada}, \cite{E.S.P. Tradacete}$)$.

\begin{definition}
Let $X$ and $W$ be Banach lattices. We denote the space of linear operators $%
T:X\rightarrow W$ for which there exists a constant $C>0$ satisfying%
\begin{equation}
\left\Vert \left( Tx_{1},...,Tx_{n}\right) \right\Vert _{W^{n},Y,\tau }\leq
C\left\Vert \left( x_{1},...,x_{n}\right) \right\Vert _{X^{n},Y,\tau },n\in 
\mathbb{N}
,\left( x_{1},...,x_{n}\right) \in X^{n}  \label{2307}
\end{equation}%
by $\Lambda _{Y}\left( X,W\right) .$ The smallest constant satisfying (\ref%
{2307}) will be indicated by $\left\Vert T\right\Vert _{\Lambda _{Y}}.$
\end{definition}

\bigskip

Since $\left\Vert \cdot \right\Vert _{W^{n},Y,\tau }$ is a norm, we have
that $\Lambda _{Y}\left( X,W\right) $ is a linear space.

\begin{lemma}
The function $\left\Vert \cdot \right\Vert _{\Lambda _{Y}}$ defines a norm
on $\Lambda _{Y}\left( X,W\right) $ and\label{PropoLambdacompleto}%
\begin{equation}
\left\Vert T\right\Vert \leq \left\Vert T\right\Vert _{\Lambda _{Y}}.
\label{Nole37}
\end{equation}
\end{lemma}

\begin{proof}
Let $T\in \Lambda _{Y}\left( X,W\right) .$ From inequality $\left( \ref%
{igualdad812}\right) ,$%
\begin{equation*}
\left\Vert Tx\right\Vert _{W}=\frac{1}{\left\Vert e_{1}\right\Vert _{Y}}%
\left\Vert Tx\right\Vert _{W^{1},Y,\tau }\leq \frac{\left\Vert T\right\Vert
_{\Lambda _{Y}}}{\left\Vert e_{1}\right\Vert _{Y}}\left\Vert x\right\Vert
_{X^{1},Y,\tau }=\left\Vert T\right\Vert _{\Lambda _{Y}}\left\Vert
x\right\Vert _{X},\forall x\in X.
\end{equation*}%
\indent Thus $\Lambda _{Y}\left( X,W\right) \subset \mathcal{L}\left(
X,W\right) $ and $\left\Vert T\right\Vert \leq \left\Vert T\right\Vert
_{\Lambda _{Y}}.$ Finally, from the respective properties of $\left\Vert
\cdot \right\Vert _{W^{n},Y,\tau }$ it follows that $\left\Vert \cdot
\right\Vert _{\Lambda _{Y}}$ defines a norm on $\Lambda _{Y}\left(
X,W\right) .$
\end{proof}

\begin{theorem}
$\Lambda _{Y}\left( X,W\right) $ is a Banach space.
\end{theorem}

\begin{proof}
Let $\left\{ T_{m}\right\} _{m=1}^{\infty }\subset $ $\Lambda _{Y}\left(
X,W\right) $ be such that $\sum_{m=1}^{\infty }\left\Vert T_{m}\right\Vert
_{\Lambda _{Y}}<\infty .$ Then 
\begin{equation*}
\sum_{m=1}^{\infty }\left\Vert T_{m}\right\Vert \leq \sum_{m=1}^{\infty
}\left\Vert T_{m}\right\Vert _{\Lambda _{Y}}<\infty
\end{equation*}%
and consequently $\sum_{m=1}^{\infty }T_{m}$ converges in $\mathcal{L}\left(
X,W\right) .$ Therefore, for each $n\in 
\mathbb{N}
$ and $x_{1},...,x_{n}\in X,$%
\begin{eqnarray*}
\left\Vert \left( \sum_{m=1}^{\infty }T_{m}x_{1},...,\sum_{m=1}^{\infty
}T_{m}x_{n}\right) \right\Vert _{W^{n},Y,\tau }\hspace{-0.5cm} &=&\hspace{%
-0.2cm}\lim_{k\rightarrow \infty }\left\Vert \left(
\sum_{m=1}^{k}T_{m}x_{1},...,\sum_{m=1}^{k}T_{m}x_{n}\right) \right\Vert
_{W^{n},Y,\tau } \\
&\leq &\lim_{k\rightarrow \infty }\sum_{m=1}^{k}\left\Vert \left(
T_{m}x_{1},...,T_{m}x_{n}\right) \right\Vert _{W^{n},Y,\tau } \\
&\leq &\sum_{m=1}^{\infty }\left\Vert T_{m}\right\Vert _{\Lambda
_{Y}}\left\Vert \left( x_{1},...,x_{n}\right) \right\Vert _{X^{n},Y,\tau }.
\end{eqnarray*}%
\indent Thus $\sum_{m=1}^{\infty }T_{m}\in \Lambda _{Y}\left( X,W\right) $
and then $\Lambda _{Y}\left( X,W\right) $ is complete.
\end{proof}

\begin{proposition}
Let $E$, $F$ be Banach spaces and $X,W$ Banach lattices\smallskip .$\label%
{propocorol}$\newline
\indent$i)$ Let $R\in \mathcal{L}\left( F,E\right) $ and $T\in \mathcal{K}%
^{Y}\left( E,X\right) .$ Then%
\begin{equation*}
T\circ R\in \mathcal{K}^{Y}\left( F,X\right) \text{ and }\left\Vert T\circ
R\right\Vert _{\mathcal{K}^{Y}}\leq \left\Vert T\right\Vert _{\mathcal{K}%
^{Y}}\left\Vert R\right\Vert .
\end{equation*}%
\indent$ii)$ Let $T\in \mathcal{K}^{Y}\left( E,X\right) $ and $L\in \Lambda
_{Y}\left( X,W\right) .$ Then 
\begin{equation*}
L\circ T\in \mathcal{K}^{Y}\left( E,W\right) \text{ and }\left\Vert L\circ
T\right\Vert _{\mathcal{K}^{Y}}\leq \left\Vert L\right\Vert _{\Lambda
_{Y}}\left\Vert T\right\Vert _{\mathcal{K}^{Y}}.
\end{equation*}%
\indent$iii)$ Let $S\in \mathcal{K}_{Y}\left( X,E\right) $ and $R\in 
\mathcal{L}\left( E,F\right) .$ Then 
\begin{equation*}
R\circ S\in \mathcal{K}_{Y}\left( X,F\right) \text{ and }\left\Vert R\circ
S\right\Vert _{\mathcal{K}_{Y}}\leq \left\Vert R\right\Vert \left\Vert
S\right\Vert _{\mathcal{K}_{Y}}.
\end{equation*}%
\indent$iv)$ Let $L\in \Lambda _{Y}\left( X,W\right) $ and $S\in \mathcal{K}%
_{Y}\left( W,E\right) .$ Then 
\begin{equation*}
S\circ L\in \mathcal{K}_{Y}\left( X,E\right) \text{ and }\left\Vert S\circ
L\right\Vert _{\mathcal{K}_{Y}}\leq \left\Vert S\right\Vert _{\mathcal{K}%
_{Y}}\left\Vert L\right\Vert _{\Lambda _{Y}}.
\end{equation*}
\end{proposition}

\begin{proof}
$i)$ Let $n\in 
\mathbb{N}
$ and $w_{1},...,w_{n}\in F.$ Then%
\begin{eqnarray*}
\left\Vert \left( T\circ R(w_{1}),...,T\circ R(w_{n}\right) \right\Vert
_{X^{n},Y,\tau }\hspace{-0.2cm} &\leq &\hspace{-0.2cm}\left\Vert
T\right\Vert _{\mathcal{K}^{Y}}\left\Vert \left( \left\Vert
Rw_{1}\right\Vert _{E},...,\left\Vert Rw_{n}\right\Vert _{E}\right)
\right\Vert _{Y} \\
&& \\
&\leq &\hspace{-0.2cm}\left\Vert T\right\Vert _{\mathcal{K}^{Y}}\left\Vert
R\right\Vert \left\Vert \left( \left\Vert w_{1}\right\Vert
_{F},...,\left\Vert w_{n}\right\Vert _{F}\right) \right\Vert _{Y}.
\end{eqnarray*}%
\indent Therefore $T\circ R\in \mathcal{K}^{Y}\left( F,X\right) $ and $%
\left\Vert T\circ R\right\Vert _{\mathcal{K}^{Y}}\leq \left\Vert
T\right\Vert _{\mathcal{K}^{Y}}\left\Vert R\right\Vert .$\newline
\newline
\indent$ii)$ Let $n\in 
\mathbb{N}
$ and $w_{1},...,w_{n}\in E.$ Then 
\begin{eqnarray*}
\left\Vert \left( L\circ T\left( w_{1}\right) ,...,L\circ T\left(
w_{n}\right) \right) \right\Vert _{W^{n},Y,\tau } &\leq &\left\Vert
L\right\Vert _{\Lambda _{Y}}\left\Vert \left( T\left( w_{1}\right)
,...,T\left( w_{n}\right) \right) \right\Vert _{X^{n},Y,\tau } \\
&& \\
&\leq &\left\Vert L\right\Vert _{\Lambda _{Y}}\left\Vert T\right\Vert _{%
\mathcal{K}^{Y}}\left\Vert \left( \left\Vert w_{1}\right\Vert
_{E},...,\left\Vert w_{n}\right\Vert _{E}\right) \right\Vert _{Y}.
\end{eqnarray*}%
\indent Therefore $L\circ T\in \mathcal{K}^{Y}\left( E,W\right) $ and $%
\left\Vert L\circ T\right\Vert _{\mathcal{K}^{Y}}\leq \left\Vert
L\right\Vert _{\Lambda _{Y}}\left\Vert T\right\Vert _{\mathcal{K}^{Y}}.$%
\newline
$\newline
$\indent The proofs of $iii)$ and $iv)$ are analogous to the proofs of $i)$
and $ii)$, respectively.
\end{proof}

\begin{corollary}
Let $X_{1},X_{2},X_{3},X_{4}$ be Banach lattices, $R\in \Lambda _{Y}\left(
X_{1},X_{2}\right) $ and $L\in \Lambda _{Y}\left( X_{3},X_{4}\right)
.\smallskip $\newline
\indent$i)$ Let $T\in \mathcal{K}^{Y}\left( X_{2},X_{3}\right) .$ Then 
\begin{equation*}
L\circ T\circ R\in \mathcal{K}^{Y}\left( X_{1},X_{4}\right)
\end{equation*}%
and%
\begin{equation*}
\left\Vert L\circ T\circ R\right\Vert _{\mathcal{K}^{Y}}\leq \left\Vert
L\right\Vert _{\Lambda _{Y}}\left\Vert T\right\Vert _{\mathcal{K}%
^{Y}}\left\Vert R\right\Vert _{\Lambda _{Y}}.\text{ }
\end{equation*}%
\newline
\indent$ii)$ Let $S\in \mathcal{K}_{Y}\left( X_{2},X_{3}\right) .$ Then 
\begin{equation*}
L\circ S\circ R\in \mathcal{K}_{Y}\left( X_{1},X_{4}\right)
\end{equation*}%
and%
\begin{equation*}
\left\Vert L\circ S\circ R\right\Vert _{\mathcal{K}^{Y}}\leq \left\Vert
L\right\Vert _{\Lambda _{Y}}\left\Vert S\right\Vert _{\mathcal{K}%
_{Y}}\left\Vert R\right\Vert _{\Lambda _{Y}}.\text{ }
\end{equation*}
\end{corollary}

\bigskip

Let us note that if the identity operator in $X$ satisfies $i\in \mathcal{K}%
^{Y}\left( X,X\right) ,$ then 
\begin{equation*}
R=i\circ R\in \mathcal{K}^{Y}\left( E,X\right) ,\forall R\in \mathcal{L}%
\left( E,X\right)
\end{equation*}%
and consequently $\mathcal{L}\left( E,X\right) \subset \mathcal{K}^{Y}\left(
E,X\right) .$ Therefore, if $X$ is $Y$-convex then%
\begin{equation*}
\mathcal{K}^{Y}\left( E,X\right) =\mathcal{L}\left( E,X\right) .
\end{equation*}%
\indent Similarly, if $X$ is $Y$-concave, then 
\begin{equation*}
\mathcal{K}_{Y}\left( X,E\right) =\mathcal{L}\left( X,E\right) .
\end{equation*}

The following class of operators is considered in \cite[Lemma 2.57]{Okada}.

\begin{definition}
Let $X$ and $W$ be Banach lattices. We denote by $\Psi _{Y}\left( X,W\right) 
$ the space of continuous linear operators $T:X\rightarrow W$ such that 
\begin{equation}
\left\Vert \left( Tx_{1},...,Tx_{n}\right) \right\Vert _{Y}\leq \left\vert
T\left( \left\Vert \left( x_{1},...,x_{n}\right) \right\Vert _{Y}\right)
\right\vert ,\forall x_{1},...,x_{n}\in X,n\in 
\mathbb{N}
.  \label{jjoo}
\end{equation}
\end{definition}

\bigskip

Observe that the above inequality is an order inequality between elements in 
$W$. Furtermore, we have that 
\begin{eqnarray}
\left\Vert \left( Tx_{1},...,Tx_{n}\right) \right\Vert _{W^{n},Y,\tau }
&=&\left\Vert \left\Vert \left( Tx_{1},...,Tx_{n}\right) \right\Vert
_{Y}\right\Vert _{W}\leq \left\Vert \left\vert T\left( \left\Vert \left(
x_{1},...,x_{n}\right) \right\Vert _{Y}\right) \right\vert \right\Vert _{W} 
\notag \\
&\leq &\left\Vert T\right\Vert \left\Vert \left( x_{1},...,x_{n}\right)
\right\Vert _{X^{n},Y,\tau },\forall T\in \Psi _{Y}\left( X,W\right) .
\label{dede}
\end{eqnarray}%
\indent Therefore $\Psi _{Y}\left( X,W\right) \subset \Lambda _{Y}\left(
X,W\right) .$

\begin{proposition}
$\mathcal{L}_{r}\left( X,W\right) \subset \Lambda _{Y}\left( X,W\right) .$
\end{proposition}

\begin{proof}
Let $P:X\rightarrow W$ be a positive operator. Fix $n\in 
\mathbb{N}
$ and $x_{1},...,x_{n}\in X.$ By Proposition $\ref{lema69},$%
\begin{equation*}
\left\Vert \left( x_{1},...,x_{n}\right) \right\Vert _{Y}\geq
\sum_{j=1}^{n}a_{j}x_{j},\forall \left( a_{1},...,a_{n}\right) \in
B_{Y^{\ast }}
\end{equation*}%
and consequently%
\begin{equation*}
P\left( \left\Vert \left( x_{1},...,x_{n}\right) \right\Vert _{Y}\right)
\geq \sum_{j=1}^{n}a_{j}Px_{j},\forall \left( a_{1},...,a_{n}\right) \in
B_{Y^{\ast }}.
\end{equation*}%
\indent Therefore, by Proposition $\ref{lema69}$,%
\begin{equation*}
\left\Vert \left( Px_{1},...,Px_{n}\right) \right\Vert _{Y}\leq P\left(
\left\Vert \left( x_{1},...,x_{n}\right) \right\Vert _{Y}\right) =\left\vert
P\left( \left\Vert \left( x_{1},...,x_{n}\right) \right\Vert _{Y}\right)
\right\vert .
\end{equation*}%
\indent Thus $P\in \Psi _{Y}\left( X,W\right) \subset $\thinspace $\Lambda
_{Y}\left( X,W\right) $ and so $\mathcal{L}_{r}\left( X,W\right) \subset
\Lambda _{Y}\left( X,W\right) .$
\end{proof}

\bigskip

Next we will prove that $Y$-convexity and $Y$-concavity are preserved under 
\textit{lattice isomorphisms}, i.e., order preserving and bijective bounded
linear operators with continuous inverse.

\begin{corollary}
Let $H_{1}:X\rightarrow W$ and $H_{2}:W\rightarrow X$ be lattice
isomorphisms. Then, for each $T\in \mathcal{L}\left( E,X\right) ,$%
\begin{equation*}
H_{1}\circ T\in \mathcal{K}^{Y}\left( E,W\right) \iff T\in \mathcal{K}%
^{Y}\left( E,X\right)
\end{equation*}
and, for each $S\in \mathcal{L}\left( X,E\right) ,$ 
\begin{equation*}
S\circ H_{2}\in \mathcal{K}_{Y}\left( W,E\right) \iff S\in \mathcal{K}%
_{Y}\left( X,E\right)
\end{equation*}
\end{corollary}

\begin{proof}
Since $H_{i},H_{i}^{-1}\in \Lambda _{Y}\left( X,W\right) $ for $i=1,2,$ the
result follows by Lemma \ref{propocorol}.
\end{proof}

\begin{corollary}
Let $H:X\rightarrow W$ be a lattice isomorphism. Then $X$ is $Y$-convex if
and only if $W$ is $Y$-convex and $X$ is $Y$-concave if and only if $W$ is $%
Y $-concave.
\end{corollary}

\begin{proof}
Since $i_{X}=H^{-1}\circ i_{W}\circ H,$ by the above corollary and Lemma \ref%
{propocorol} it follows that%
\begin{align*}
i_{X}\in \mathcal{K}^{Y}\left( X,X\right) & \iff H\circ i_{X}=i_{W}\circ
H\in \mathcal{K}^{Y}\left( X,W\right) \\
& \iff i_{W}\circ H\circ H^{-1}=i_{W}\in \mathcal{K}^{Y}\left( W,W\right)
\end{align*}%
and%
\begin{align*}
i_{X}\in \mathcal{K}_{Y}\left( X,X\right) & \iff i_{X}\circ
H^{-1}=H^{-1}\circ i_{W}\in \mathcal{K}_{Y}\left( W,X\right) \\
& \iff H\circ H^{-1}\circ i_{W}=i_{W}\in \mathcal{K}_{Y}\left( W,W\right) .
\end{align*}
\end{proof}

\bigskip

Observe that, from the above results we obtain the following containments: 
\begin{equation*}
\text{Lattice isomorphisms}\subset \mathcal{L}_{r}\left( X,W\right) \subset
\Lambda _{Y}\left( X,W\right) \subset \mathcal{L}\left( X,W\right) .
\end{equation*}%
\linebreak

\section{$Y$-summability}

In this section we present a generalization of the classical space of 
\textit{absolutely }$p$\textit{-summing operators }$\Pi _{p}$ (\cite{Diestel}%
, \cite[Ch. 1.d]{Linden}, \cite[p. 79]{Okada}, \cite{A class of summing}).
Recall $E,F$ are Banach spaces, $X,W,Z$ are Banach lattices and $Y$ is a
Banach sequence lattice.

\bigskip

For each $n\in 
\mathbb{N}
$ we now define the function $\left\Vert \cdot \right\Vert
_{w,X^{n},Y}:X^{n}\rightarrow 
\mathbb{R}
$ by%
\begin{equation*}
\left\Vert \left( x_{1},...,x_{n}\right) \right\Vert
_{w,X^{n},Y}:=\sup_{x^{\ast }\in B_{X^{\ast }}}\left\{ \left\Vert \left(
\left\langle x_{1},x^{\ast }\right\rangle ,...,\left\langle x_{n},x^{\ast
}\right\rangle \right) \right\Vert _{Y}\right\} ,\forall \left(
x_{1},...,x_{n}\right) \in X^{n}.
\end{equation*}%
\indent Clearly $\left\Vert \cdot \right\Vert _{w,X^{n},Y}$ is a norm on $%
X^{n}.$

\begin{definition}
Let $X$ be a Banach lattice, $Y$ a Banach sequence lattice and $E$ a Banach
space. We say that an operator $S:X\rightarrow E$ is absolutely $Y$\textit{%
-summing} if there exists a constant $C>0$ such that, for each $n\in 
\mathbb{N}
$ and $x_{1},...,x_{n}\in X,$%
\begin{equation}
\left\Vert \left( Sx_{1},...,Sx_{n}\right) \right\Vert _{E^{n},Y}\leq
C\left\Vert \left( x_{1},...,x_{n}\right) \right\Vert _{w,X^{n},Y}.
\label{dessss}
\end{equation}%
\indent The space of absolutely $Y$\textit{-summing operators from }$X$ into 
$E$ will be denoted by $\Pi _{Y}\left( X,E\right) $ and the smallest
constant satisfying (\ref{dessss}) by $\left\Vert S\right\Vert _{\Pi
_{Y}\left( X,E\right) }$ or simply by $\left\Vert S\right\Vert _{\Pi _{Y}}.$
\end{definition}

\bigskip

Clearly $\Pi _{Y}\left( X,E\right) $ is a linear space and, since $%
\left\Vert \cdot \right\Vert _{E^{n},Y}$ is a norm, it follows that $%
\left\Vert \cdot \right\Vert _{\Pi _{Y}\left( X,E\right) }$ is a norm on $%
\Pi _{Y}\left( X,E\right) .$

\begin{lemma}
$\Pi _{Y}\left( X,E\right) $ is continuously included in $\mathcal{K}%
_{Y}\left( X,E\right) \label{propointro}$
\end{lemma}

\begin{proof}
Let $S\in \Pi _{Y}\left( X,E\right) .$ Observe that for each $x^{\ast }\in
X^{\ast }$ the functional $\left\vert x^{\ast }\right\vert $ is positive.
Thus $\left\vert x^{\ast }\right\vert \in \Psi _{Y}\left( X,%
\mathbb{R}
\right) $ and satisfies inequality $\left( \ref{positive}\right) .$ Then for
each $n\in 
\mathbb{N}
$ and $x_{1},...,x_{n}\in X,$%
\begin{eqnarray*}
\left\Vert \left( Sx_{1},...,Sx_{n}\right) \right\Vert _{E^{n},Y} &\leq
&\left\Vert S\right\Vert _{\Pi _{Y}}\sup_{x^{\ast }\in B_{X^{\ast }}}\left\{
\left\Vert \left( \left\langle x_{1},x^{\ast }\right\rangle
,...,\left\langle x_{n},x^{\ast }\right\rangle \right) \right\Vert
_{Y}\right\} \\
&\leq &\left\Vert S\right\Vert _{\Pi _{Y}}\sup_{x^{\ast }\in B_{X^{\ast
}}}\left\{ \left\Vert \left( \left\langle \left\vert x_{1}\right\vert
,\left\vert x^{\ast }\right\vert \right\rangle ,...,\left\langle \left\vert
x_{n}\right\vert ,\left\vert x^{\ast }\right\vert \right\rangle \right)
\right\Vert _{Y}\right\} \\
&\leq &\left\Vert S\right\Vert _{\Pi _{Y}}\sup_{x^{\ast }\in B_{X^{\ast
}}}\left\{ \left\vert x^{\ast }\right\vert \left( \left\Vert \left(
\left\vert x_{1}\right\vert ,...,\left\vert x_{n}\right\vert \right)
\right\Vert _{Y}\right) \right\} \\
&=&\left\Vert S\right\Vert _{\Pi _{Y}}\left\Vert \left\Vert \left(
\left\vert x_{1}\right\vert ,...,\left\vert x_{n}\right\vert \right)
\right\Vert _{Y}\right\Vert _{X} \\
&=&\left\Vert S\right\Vert _{\Pi _{Y}}\left\Vert \left(
x_{1},...,x_{n}\right) \right\Vert _{X^{n},Y,\tau }.
\end{eqnarray*}%
\indent Therefore $S\in \mathcal{K}_{Y}\left( X,E\right) $ and 
\begin{equation}
\left\Vert S\right\Vert _{\mathcal{K}_{Y}}\leq \left\Vert S\right\Vert _{\Pi
_{Y}}.  \label{a}
\end{equation}
\end{proof}

\begin{theorem}
$\Pi _{Y}\left( X,E\right) $ is a Banach space
\end{theorem}

\begin{proof}
Let $\left\{ S_{m}\right\} _{m=1}^{\infty }\subset $ $\Pi _{Y}\left(
X,E\right) $ be such that $\sum_{m=1}^{\infty }\left\Vert S_{m}\right\Vert
_{\Pi _{Y}}<\infty .$ By inequality $\left( \ref{a}\right) ,$ 
\begin{equation*}
\sum_{m=1}^{\infty }\left\Vert S_{m}\right\Vert _{\mathcal{K}_{Y}}\leq
\sum_{m=1}^{\infty }\left\Vert S_{m}\right\Vert _{\Pi _{Y}}<\infty
\end{equation*}%
and consequently $\sum_{m=1}^{\infty }S_{m}$ converges in $\mathcal{K}%
_{Y}\left( X,E\right) .$ Then, for each $n\in 
\mathbb{N}
$ and $x_{1},...,x_{n}\in X,$%
\begin{eqnarray*}
\left\Vert \left( \sum_{m=1}^{\infty }S_{m}x_{1},...,\sum_{m=1}^{\infty
}S_{m}x_{n}\right) \right\Vert _{E^{n},Y}\hspace{-0.5cm} &=&\lim_{k%
\rightarrow \infty }\left\Vert \left(
\sum_{m=1}^{k}S_{m}x_{1},...,\sum_{m=1}^{k}S_{m}x_{n}\right) \right\Vert
_{E^{n},Y} \\
&\leq &\lim_{k\rightarrow \infty }\sum_{m=1}^{k}\left\Vert \left(
S_{m}x_{1},...,S_{m}x_{n}\right) \right\Vert _{E^{n},Y} \\
&\leq &\lim_{k\rightarrow \infty }\left( \sum_{m=1}^{k}\left\Vert
S_{m}\right\Vert _{\Pi _{Y}}\left\Vert \left( x_{1},...,x_{n}\right)
\right\Vert _{w,X^{n},Y}\right) \\
&=&\left( \sum_{m=1}^{\infty }\left\Vert S_{m}\right\Vert _{\Pi _{Y}}\right)
\left\Vert \left( x_{1},...,x_{n}\right) \right\Vert _{w,X^{n},Y}.
\end{eqnarray*}%
\indent Therefore $\sum_{m=1}^{\infty }S_{m}\in \Pi _{Y}\left( X,E\right) $
and $\Pi _{Y}\left( X,E\right) $ is complete.
\end{proof}

\begin{proposition}
Every finite rank operator $S:X\rightarrow E$ is absolutely $Y$-summing.%
\label{0709}
\end{proposition}

\begin{proof}
Let $S:X\rightarrow E$ a rank $1$ operator$.$ Take $\varphi \in X^{\ast }$
and $w\in E$ such that $S\left( x\right) =\left\langle x,\varphi
\right\rangle w,x\in X.$ Then, for every $n\in 
\mathbb{N}
$ and $x_{1},...,x_{n}\in X,$%
\begin{eqnarray*}
\left\Vert \left( Sx_{1},...,Sx_{n}\right) \right\Vert _{E^{n},Y}
&=&\left\Vert \left( \left\Vert \left\langle x_{1},\varphi \right\rangle
w\right\Vert _{E},...,\left\Vert \left\langle x_{n},\varphi \right\rangle
w\right\Vert _{E}\right) \right\Vert _{Y} \\
&& \\
&=&\left\Vert w\right\Vert _{E}\left\Vert \left( \left\langle x_{1},\varphi
\right\rangle ,...,\left\langle x_{n},\varphi \right\rangle \right)
\right\Vert _{Y} \\
&& \\
&\leq &\left\Vert w\right\Vert _{E}\left\Vert \varphi \right\Vert \sup
\left\{ \left\Vert \left( \left\langle x_{1},x^{\ast }\right\rangle
,...,\left\langle x_{n},x^{\ast }\right\rangle \right) \right\Vert
_{Y}:x^{\ast }\in B_{X^{\ast }}\right\} .
\end{eqnarray*}%
\indent So $S\in \Pi _{Y}\left( X,E\right) .$ The result follows since every
finite rank operator can be expressed as a linear combination of rank $1$
operators.
\end{proof}

\bigskip

By Lemma $\ref{propointro}$ and Proposition \ref{0709} we have the following
corollary.

\begin{corollary}
Let $X$ be a Banach lattice and $E$ a Banach space$.$ Then every finite rank
operator $S:X\rightarrow E$ is $Y$-concave.
\end{corollary}

\begin{proposition}
Let $X,W$ be Banach lattices, $E,F$ Banach spaces, $T\in \mathcal{L}\left(
E,F\right) ,$ $S\in \Pi _{Y}\left( W,E\right) $ and $R\in \mathcal{L}\left(
X,W\right) .$ Then $T\circ S\circ R\in \Pi _{Y}\left( X,F\right) $ and%
\begin{equation}
\left\Vert T\circ S\circ R\right\Vert _{\Pi _{Y}\left( X,F\right) }\leq
\left\Vert T\right\Vert \left\Vert S\right\Vert _{\Pi _{Y}\left( W,E\right)
}\left\Vert R\right\Vert .  \label{cff3}
\end{equation}
\end{proposition}

\begin{proof}
Consider $n\in 
\mathbb{N}
$, take $x_{1},...,x_{n}\in X$ and denote by $c$ the expression\ $\left\Vert
\left( T\circ S\circ Rx_{1},...,T\circ S\circ Rx_{n}\right) \right\Vert
_{F^{n},Y}.$ Then 
\begin{eqnarray}
c &=&\left\Vert \left( \left\Vert T\circ S\circ Rx_{1}\right\Vert
_{F},...,\left\Vert T\circ S\circ Rx_{n}\right\Vert _{F}\right) \right\Vert
_{Y}  \notag \\
&&  \notag \\
&\leq &\left\Vert T\right\Vert \left\Vert \left( \left\Vert S\circ
Rx_{1}\right\Vert _{E},...,\left\Vert S\circ Rx_{n}\right\Vert _{E}\right)
\right\Vert _{Y}  \notag \\
&&  \notag \\
&=&\left\Vert T\right\Vert \left\Vert \left( S\circ Rx_{1},...,S\circ
Rx_{n}\right) \right\Vert _{E^{n},Y}  \notag \\
&&  \notag \\
&\leq &\left\Vert T\right\Vert \left\Vert S\right\Vert _{\Pi _{Y}\left(
W,E\right) }\left\Vert \left( Rx_{1},...,Rx_{n}\right) \right\Vert
_{w,W^{n},Y}.  \label{cff}
\end{eqnarray}%
\indent Observe now that, for each $w^{\ast }\in B_{W^{\ast }},$ we have
that $w^{\ast }\circ R\in X^{\ast }$ and $\left\Vert w^{\ast }\circ
R\right\Vert \leq \left\Vert R\right\Vert .$ Thus 
\begin{eqnarray}
\left\Vert \left( Rx_{1},...,Rx_{n}\right) \right\Vert _{w,W^{n},Y}
&=&\sup_{w^{\ast }\in B_{W^{\ast }}}\left\{ \left\Vert \left( \left\langle
Rx_{1},w^{\ast }\right\rangle ,...,\left\langle Rx_{n},w^{\ast
}\right\rangle \right) \right\Vert _{Y}\right\}  \notag \\
&&  \notag \\
&=&\sup_{w^{\ast }\in B_{W^{\ast }}}\left\{ \left\Vert \left( \left\langle
x_{1},w^{\ast }\circ R\right\rangle ,...,\left\langle x_{n},w^{\ast }\circ
R\right\rangle \right) \right\Vert _{Y}\right\}  \notag \\
&&  \notag \\
&\leq &\sup_{x^{\ast }\in X^{\ast }}\left\{ \left\Vert \left( \left\langle
x_{1},x^{\ast }\right\rangle ,...,\left\langle x_{n},x^{\ast }\right\rangle
\right) \right\Vert _{Y}:\left\Vert x^{\ast }\right\Vert \leq \left\Vert
R\right\Vert \right\}  \notag \\
&&  \notag \\
&=&\left\Vert R\right\Vert \sup_{x^{\ast }\in B_{X^{\ast }}}\left\{
\left\Vert \left( \left\langle x_{1},x^{\ast }\right\rangle
,...,\left\langle x_{n},x^{\ast }\right\rangle \right) \right\Vert
_{Y}\right\} .  \label{cff2}
\end{eqnarray}%
\indent From $\left( \ref{cff}\right) $ and $\left( \ref{cff2}\right) $ it
follows that $T\circ S\circ R\in \Pi _{Y}\left( X,F\right) $ and $\left( \ref%
{cff3}\right) $ holds.
\end{proof}

\end{document}